\documentclass[a4paper,10pt,reqno]{amsart}
\usepackage{etex}
\usepackage{amssymb}
\usepackage{amsmath}
\usepackage{amsfonts}
\usepackage{amsthm}
\usepackage{latexsym}
\usepackage{mathrsfs}
\usepackage{eucal}
\usepackage[safe]{tipa}
\usepackage{cases}
\usepackage{scalerel}
\usepackage{xcolor}

\newcommand{\red}[1]{{\color{red}{#1}}}

\usepackage{comment}
\excludecomment{TOOLONG}

\usepackage{slashed}
\usepackage{textalpha}
\usepackage{lpic}
\usepackage{tikz}
\usetikzlibrary{cd}
\usetikzlibrary{arrows} 
\usetikzlibrary{graphs}
\usetikzlibrary{intersections}
\usetikzlibrary{positioning}
\usetikzlibrary{decorations.pathmorphing}
\usetikzlibrary{datavisualization} 
\usetikzlibrary{datavisualization.formats.functions}
\usetikzlibrary{intersections}
\usetikzlibrary{hobby}
\usetikzlibrary{through}
\usetikzlibrary{calc}

\usepackage{pgfmath}
\usepackage{pgfkeys}

\usetikzlibrary{automata,positioning}
\usetikzlibrary{decorations.markings}
\tikzset{->-/.style={decoration={
  markings,
  mark=at position #1 with {\arrow{>}}},postaction={decorate}}}
  \tikzset{-<-/.style={decoration={
  markings,
  mark=at position #1 with {\arrowreversed[red]{latex'}}},postaction={decorate}}}

\usepackage{url}
\usepackage[linktocpage=true]{hyperref}
\hypersetup{
    colorlinks=true,
    linkcolor=blue,
    filecolor=blue,      
    urlcolor=blue
}

\newcommand{\la}{\langle}
\newcommand{\ra}{\rangle}
\newcommand{\Reals}{\mathbb{R}}
\newcommand{\nut}{P}
\newcommand{\bolt}{B}
\newcommand{\sgn}{\mathop{\rm sgn}}
\newcommand{\sign}{\mathop{\rm sign}}

\newcommand{\Riem}{\mathop{\rm Riem}}

\newcommand{\Scal}{\mathop{\rm Scal}}

\newcommand{\cE}{\mathcal{E}}
\newcommand{\cB}{\mathcal{B}}

\newcommand{\cZ}{\mathcal{Z}}

\newcommand{\cR}{\mathcal{R}}
\newcommand{\cM}{\mathcal{M}}
\newcommand{\cN}{\mathcal{N}}
\newcommand{\cF}{\mathcal{F}}
\newcommand{\wh}{\widehat}
\newcommand{\Th}{\Theta}
\newcommand{\al}{\alpha}
\newcommand{\be}{\beta}

\newcommand{\si}{\sigma}

\theoremstyle{plain}
\newtheorem{thm}{Theorem}[section]
\newtheorem{cor}[thm]{Corollary}
\newtheorem{lemma}[thm]{Lemma}

\newtheorem{definition}[thm]{Definition}
\newtheorem{prop}[thm]{Proposition}
\newtheorem{conj}{Conjecture}

\newtheorem{remark}[thm]{Remark}

\renewcommand{\d}{\mathrm{d}}

\DeclareMathOperator{\CP}{\mathbb{C}\mathbb{P}}
\DeclareMathOperator{\RP}{\mathbb{R}\mathbb{P}}

\newcommand{\SU}{\text{SU}}
\newcommand{\nnuts}{n_{\text{\rm nuts}}}
\newcommand{\nbolts}{n_{\text{\rm bolts}}} 
\newcommand{\GenVec}{\nu}
\newcommand{\Lie}{\mathcal{L}}
\newcommand{\Mring}{\mathring{M}}
\newcommand{\cMring}{\mathring{\cM}}
\newcommand{\gring}{\mathring{g}}

\newcommand{\Zbb}{\mathbb{Z}}

\renewcommand{\P}[1]{{{#1}^{\hspace{-.5pt}\scaleto{+}{4pt}}}}
\newcommand{\M}[1]{{{#1}^{\hspace{-.5pt}\scaleto{-}{4pt}}}}
\newcommand{\PM}[1]{{{#1}^{\hspace{-.5pt}\scaleto{\pm}{3.5pt}}}}
\newcommand{\MP}[1]{{{#1}^{\hspace{-.5pt}\scaleto{\mp}{3.5pt}}}}
\newcommand{\eps}{\epsilon}

\newcommand{\Fixed}{F}

\def\Wcal{\mathcal{W}}
\def\Fcal{\mathcal{F}}
\def\Ecal{\mathcal{E}}
\def\Qcal{\mathcal{Q}}
\def\Scal{\mathcal{S}}
\def\Ical{\mathcal{I}}
\def\sbold{\mathbf{s}}

\title{Gravitational instantons with $S^1$ symmetry} 
\author[S. Aksteiner]{Steffen Aksteiner}
\email{steffen.aksteiner@aei.mpg.de}
\address{Albert Einstein Institute, Am M\"uhlenberg 1, D-14476 Potsdam, Germany }
\author[L. Andersson]{Lars Andersson}
\email{lars.andersson@bimsa.cn}
\address{Beijing Institute of Mathematical Sciences and Applications, Beijing 101408, China}
\author[M. Dahl]{Mattias Dahl}
\email{dahl@math.kth.se}
\address{Institutionen f\"or Matematik, Kungliga Tekniska H\"ogskolan, 100 44 Stockholm, Sweden} 
\author[G. Nilsson]{Gustav Nilsson}
\email{gustav.nilsson@aei.mpg.de}
\address{Albert Einstein Institute, Am M\"uhlenberg 1, D-14476 Potsdam, Germany }
\author[W. Simon]{Walter Simon}
\email{walter.simon@univie.ac.at}
\address{Fakult\"at f\"ur Mathematik, Universit\"at Wien, Oskar-Morgenstern Platz 1, 1090 Vienna, Austria}

\allowdisplaybreaks[2]

\numberwithin{equation}{section}

\begin{document}

\begin{abstract}
Uniqueness results for asymptotically locally flat and asymptotically flat $S^1$-symmetric gravitational instantons are proved using a divergence identity of the type used in uniqueness proofs for static black holes, combined with results derived from the $G$-signature theorem. Our results include a proof of the $S^1$-symmetric version of the Euclidean Black Hole Uniqueness conjecture, a uniqueness result for the Taub-bolt family of instantons, as well as a proof that an ALF $S^1$-symmetric instanton with the topology of the Chen--Teo family of instantons is Hermitian.  
\end{abstract} 

\maketitle
\tableofcontents

\section{Introduction}\label{sec:intro}

A complete Ricci-flat four-manifold with at least quadratic curvature decay is called a gravitational instanton. Apart from their intrinsic interest as geometric objects, the study of gravitational instantons is motivated by ideas from Yang--Mills theory and Quantum Gravity \cite{hawking:1977:grav:instantons}\footnote{Although the notion of gravitational instanton is sometimes introduced with additional assumptions, for example that the space is hyperk\"ahler, we make no such a priori assumptions in this paper.}. 
In this paper we contribute to the classification of instantons by proving a set of uniqueness results for asymptotically locally flat (ALF) $S^1$-symmetric instantons, making assumptions only on their topology. One of our results is the $S^1$-symmetric version of the Riemannian signature version of the Black Hole Uniqueness conjecture of General Relativity, known as the Euclidean Black Hole Uniqueness Conjecture, which states that an instanton on $S^4\setminus S^1$ is in the Kerr family of gravitational instantons, the Wick-rotated version of the Kerr family of black hole solutions of General Relativity.  We shall focus on ALF instantons that admit an effective $S^1$-action by isometries, generated by a Killing field with uniformly bounded norm, and use the term $S^1$-instantons for these, cf.  Definition \ref{def:S1inst} below.

Some early results related to the present work were presented by one of the authors in \cite{simon:1995,mars:simon:1999}. In those papers,  additional assumptions on the structure of the fixed point set were needed, as well as some non-trivial technical assumptions that are removed here. The facts about the structure of the fixed point set that were used in the just mentioned papers are applications of the index theorems  to instantons with $S^1$ symmetry, as presented by Gibbons and Hawking in \cite{gibbons:hawking:1979}, see also \cite{1979NuPhB.157..377G}. In this paper, we instead follow the approach of Jang. See \cite{2018JGP...133..181J} and references therein. This method, which makes use of the full power of the $G$-signature formula, gives more precise information and enables us to eliminate assumptions on the fixed point set of the $S^1$-action.

Recall that an ALF instanton has cubic volume growth and boundary at infinity $L$ of topology $S^1\times S^2$, $S^3$, or a quotient thereof, see Definition \ref{def:BG}. The case when $L$ is a circle bundle over $S^2$ or $\RP^2$ is called ALF-$A_k$ or ALF-$D_k$, respectively. In the case of  ALF-$A_k$ instantons, $k = - e -1$, where $e$ is the Euler characteristic of the circle bundle over $S^2$. The case ALF-$A_{-1}$ when the boundary at infinity is a trivial circle bundle over $S^2$ is called asymptotically flat (AF). 

The most well-known AF instanton is the Euclidean Kerr Instanton, a 2-parameter family of instantons on  $S^4 \setminus S^1 \cong S^2 \times \Reals^2$, with the Wick-rotated  Kerr metric \cite{1977PhRvD..15.2752G,1978RSPSA.358..467G,1998PhRvD..59b4009H}. 

The Euclidean Kerr Instanton is toric, that is it admits an effective $T^2$-action by isometries, and is algebraically special, of Petrov type $D^+D^-$, but does not have special holonomy. In addition to the Kerr Instanton, there is the remarkable Chen--Teo Instanton \cite{2011PhLB..703..359C}, a 2-parameter family of toric AF instantons on $\CP^2 \setminus S^1$. Two of the authors have recently proved that the Chen--Teo Instanton is one-sided algebraically special \cite{MR4809333}, and hence also Hermitian \cite{MR707181}, i.e. it admits an integrable complex structure compatible with the metric. In view of this fact, all known examples of instantons are Hermitian. Both the  Kerr and Chen--Teo instantons are $S^1$-symmetric for suitable values of the parameters satisfying a rationality condition\footnote{For general parameter values, both Kerr and Chen--Teo admit a Killing field with bounded norm, which however generates closed orbits only for a restricted set of parameters.}. 

The Chen--Teo instanton is a counterexample to one of the 1970's era uniqueness conjectures for gravitational instantons, the Euclidean Black Hole Uniqueness Conjecture \cite[Conjecture 2]{gibbons:1978:survey}, which stated that a non-flat AF instanton is in the Kerr family\footnote{Note that the Kerr family includes the spherically symmetric Schwarzschild Instanton, and that a flat ALF instanton is, up to a rescaling,  isometric to $\Reals^3 \times S^1$.}. This shows that the Euclidean Black Hole Uniqueness Conjecture can hold only with additional assumptions, and moreover indicates that the classification problem for gravitational instantons merits renewed attention. 

Examples of ALF-$A_k$ instantons include the hyperk\"ahler Taub--NUT instanton \cite{hawking:1977:grav:instantons,lebrun:1989} with topology $S^4 \setminus \{\text{\rm pt.}\}$ and the Petrov type $D^+D^-$ Taub-bolt instanton \cite{1978PhLB...78..249P} with topology $\CP^2 \setminus \{\text{\rm pt.}\}$. These are both $\SU(2)$-symmetric, and in particular toric and $S^1$-symmetric. The remaining ALF-$A_k$ examples are the Gibbons--Hawking multi-Taub--NUT solutions and the ALF-$D_k$ instantons with the $S^1$-symmetric Atiyah--Hitchin metrics, and the Cherkis--Hitchin--Kapus\-tin--Ivanov--Lindstr\"om--Ro\v{c}ek metrics, see \cite{MR2177322,chen2019gravitational}, which are all hyperkähler. See also \cite{MR2855540}. The collinear multi-Taub--NUT instantons are toric and $S^1$-symmetric. 

For gravitational instantons with special geometry, some classification results are known. There is a complete classification of  hyperk\"ahler instantons with total curvature in $L^2$ \cite{chen2015gravitational,chen2019gravitational,chen2021gravitational,2021arXiv210812991S}. Further, toric ALF Hermitian instantons were recently classified by Biquard and Gauduchon \cite{Biquard:Gauduchon}, and shown to belong to one of the Kerr, Chen--Teo, (multi-)Taub--NUT, or Taub-bolt families. The general classification problem remains open. See also \cite{2021LMaPh.111..133K} for uniqueness results on toric instantons without assumption on special geometry. 

The following theorem states the main results of this paper.  
\begin{thm}\label{thm:main-intro} Let $(\cM, g_{ab})$ be an $S^1$-instanton of ALF-$A_k$ type. Then the following holds. 
\begin{enumerate} 
\item \label{point:Kerr} 
If $\cM \cong S^4 \setminus S^1$, then $(\cM, g_{ab})$ is in the Kerr family of instantons.
\item \label{point:T-b} 
If $\cM \cong \CP^2 \setminus \{\text{\rm pt.}\}$, then $(\cM, g_{ab})$ is in the Taub-bolt family of instantons.
\item \label{point:3} 
If $\cM \cong \CP^2 \setminus S^1$, then $(\cM, g_{ab})$ is Hermitian. 
\end{enumerate} 
\end{thm} 
Point \ref{point:3} of Theorem \ref{thm:main-intro} is not a rigidity result in the same sense as the first two points. However, it can be viewed as a step towards rigidity for the Chen--Teo family in view of the following conjecture.   
\begin{conj}[\protect{\cite[Conjecture 1]{MR4809333}}] \label{conj:Herm-class} Let $(\cM, g_{ab})$ be a non-K\"ahler ALF  Hermitian instanton. Then $(\cM, g_{ab})$ is one of Kerr, Chen--Teo, Taub-bolt, or Taub--NUT with the orientation opposite to the hyperk\"ahler orientation. 
\end{conj} 
\begin{remark} 
An ALF non-K\"ahler Hermitian instanton admits a bounded Killing field, and is locally conformal to an extremal K\"ahler space. By analogy with the classification of compact Hermitian-Einstein spaces \cite{MR2899877},  it is plausible that a non-K\"ahler Hermitian instanton is actually toric, which in view of the classification result of \cite{Biquard:Gauduchon} would imply the conjecture.  
\end{remark} 
We further make the following progressively stronger conjectures, closely related to Conjecture \ref{conj:Herm-class}. 
\begin{conj} \label{conj:KVF-class} The conclusion of Conjecture \ref{conj:Herm-class} holds for a non-K\"ahler ALF gravitational instanton which admits a non-vanishing Killing field with bounded norm.
\end{conj}
\begin{remark} 
In view of the results of  Yau \cite{MR448379}, the existence of one Killing field with bounded norm for a non-flat ALF instanton implies that one has either  $S^1$ or toric symmetry. See also \cite[\S 2]{gibbons:hawking:1979}. The results in the present paper imply that an ALF instanton with the topology of Kerr, Taub-bolt or Chen-Teo must be Hermitian. Therefore the only part of the conjecture that is open is the case of toric instantons without special geometry. 
\end{remark}
Finally, we mention the following conjecture, analogous to those made in \cite{gibbons:1978:survey}, and stronger than Conjectures \ref{conj:Herm-class} and \ref{conj:KVF-class}. 
\begin{conj} \label{conj:strong-class}  
An ALF gravitational instanton is Hermitian. 
\end{conj} 
\begin{remark}  
The problem raised by Conjecture \ref{conj:strong-class} can be expected to be difficult, analogous to the question whether all compact Ricci-flat spaces have special holonomy, see \cite[p. 19]{MR2371700}. However, the ALF assumption provides structure not present in the compact case.  
\end{remark} 
The Chen--Teo Instanton was found using the Belinski--Zakha\-rov soliton method, dressing a flat seed solution with four solitons and coalescing one pair of those. The computational methods employed by Chen and Teo did not enable them to explore solutions with more than four solitons. Thus, a priori there could exist additional instantons which can be constructed along the same lines, but containing more solitons. This is an interesting direction of research. 

In addition to the families of instantons mentioned above, some additional potential examples have been mentioned in the literature. These include the Euclidean signature of the Kerr--NUT family of metrics, for example the so-called Kerr--Taub-bolt instanton, see \cite[\S III]{gibbons:perry:1980}. An argument showing that the Euclidean Kerr--Taub-bolt  instanton must have line singularities was given in \cite[\S 4.8] {2010NuPhB.838..207C}, see also \cite[Appendix A]{clement2020smarr}.

We shall now indicate some of the background for the technique used in this paper. Recall that the Schwarzschild and Kerr instantons were constructed by Wick rotating known Lorentzian black hole metrics. The problem of classifying gravitational instantons has important parallels to the Black Hole Uniqueness Problem, and indeed, the divergence identity used in this paper arose out of attempts to prove the Black Hole Uniqueness Conjecture by  generalizing Israel's method \cite{1967PhRv..164.1776I,MR398432,1977GReGr...8..695R} for the uniqueness problem from the static Lorentzian case to the stationary Lorentzian case. 

In the case of a hypersurface orthogonal Killing vector field $\xi$ this method applies to spaces of both Riemannian and Lorentzian signature (with non-compact isometry group $\mathbb{R}$ in the Lorentzian case). The key ingredients in the argument are divergence identities on the 3-dimensional orbit space, among which we select
\begin{align} \label{eq:Israel} 
{}^{(3)}\nabla_i 
\left( \frac{(1 + \lambda)^2}{\sqrt{\lambda}} {}^{(3)}\nabla^i \frac{\sqrt{W}}{1 - \lambda^2}\right) = Q,    
\end{align} 
which is closely related for example to \cite[Eqs. (2.12)-(2.13)]{MR398432}. Here ${}^3\nabla_i$ is the Levi-Civita covariant derivative on the orbit space, $\lambda$ is the squared norm of the Killing field, $\lambda  = |g_{ab}\xi^a \xi^b|$,  $W$ is given by $W = ({}^{(3)}\nabla_i  \sqrt{\lambda})({}^{(3)} \nabla^i \sqrt{\lambda})$, and $Q$ is a non-negative expression in curvature such that $Q = 0$ if and only if the orbit space is conformally flat. The integral of the left hand side of \eqref{eq:Israel} can, after dealing with potential singularities resulting from possible  zeros of $W$, be evaluated in terms of contributions from infinity and from the surface gravities computed from the fixed point data of the isometry. The latter, in turn, can be related to the topology of $\cM$ using index theorems. In case the left hand side vanishes, it is possible to conclude that the space(-time) under consideration is Schwarzschild. This perspective is further developed in Appendix \ref{sec:quotient}.

The above scheme can, with considerable effort, be generalized to the generic case when $\xi$ fails to be hypersurface-orthogonal. Let $\PM{\Fcal}_{ab}$ be the (anti)-self dual parts of $(d\xi)_{ab}$, let $\PM{\sigma}_{a}= 2 \PM{\Fcal}_{ab} \xi^{b}$, and define the Ernst potentials $\PM{\Ecal}$ by $\nabla_a \PM{\Ecal} = \PM{\sigma}_a$. In the Riemannian case, this yields a pair of real scalar fields $\PM{\Ecal}$, while in the Lorentzian case, we have the familiar complex Ernst potential and its complex conjugate.  

Restricting ourselves to the  Riemannian case, we define real currents $\PM{\Psi}_a$ in terms of $\PM{\Ecal}$, $\lambda$, and the norms of $\PM{\Fcal}_{ab}$, see Section \ref{sec:divident}, such that  
\begin{align} \label{eq:Simon} 
\nabla_a \PM{\Psi}^a = \PM{\Qcal}. 
\end{align} 
Here $\PM{\Qcal}$ are non-negative expressions defined in terms of $\lambda$, the Weyl tensor, and $\PM{\Fcal}_{ab}$, and such that $\PM{\Qcal} = 0$ if and only if $(\cM, g_{ab})$ is algebraically special. This pair of currents  was first introduced by one of the authors \cite{simon:1995}. As in the case when $\xi$ is hypersurface orthogonal, when \eqref{eq:Simon} reduces to \eqref{eq:Israel}, the integral of the left hand side of \eqref{eq:Simon} can be evaluated in terms of the fixed point data for the $S^1$ action and the length of the $S^1$-orbits at infinity. 

The $G$-signature theorem for an ALF-$A_k$ instanton yields an identity which equates the signature of $\cM$ with a rational expression, with parameters determined by the fixed point data for the $S^1$-action, including the orientation and weights of the nuts, the self-intersection numbers of the bolts, as well as the Euler characteristic of the circle bundle at infinity, see \eqref{eq:G-sign}. The fact that this rational expression is actually constant implies strong restrictions on the fixed point data. This observation was used by Jang, see \cite{2018JGP...133..181J}, following earlier work by Li and Liu \cite{2010arXiv1012.1507L}, and Li \cite{MR3343967} to derive strong conditions on the fixed point set in terms of the topology, in the case of $S^1$-actions with only isolated fixed points. See also \cite{MR1201437} for earlier related work. 

Here we generalize these ideas to give restrictions on the fixed point set for ALF-$A_k$ spaces with an $S^1$-action, without assuming that there are only isolated fixed points.  Making use of the results on the structure of the fixed point set derived from the $G$-signature theorem in the manner mentioned above, it is possible to show that for suitable topologies, equality holds in \eqref{eq:Simon}, which then implies algebraic speciality. 

In the particular cases of instantons with topology $S^4\setminus S^1$ one has equality for both signs in \eqref{eq:Simon}, which implies that the space is Petrov type $D^+D^-$. This, in turn, by \cite{Biquard:Gauduchon}  implies that the space is in the Euclidean Kerr family. The same situation holds in the case of $\CP^2 \setminus \{\text{\rm pt.}\}$, yielding a uniqueness result for the Taub-bolt instanton. The just discussed results are points \ref{point:Kerr}, \ref{point:T-b} of Theorem \ref{thm:main-intro}. Finally, for an instanton of topology $\CP^2 \setminus S^1$, one finds that equality holds only for one sign, yielding the statement that the instanton is Hermitian, point \ref{point:3} of Theorem \ref{thm:main-intro}.  

We end this introduction by pointing out that in the Lorentzian case, the (anti)-self dual parts $\PM{\Fcal}$ of $d\xi$ are complex, which means that the Ernst potential, as well as the corresponding current $\Psi_a$, are complex. Therefore, although there is a complex counterpart to \eqref{eq:Simon} in the Lorentzian case, see \cite{MR747036}, it is not clear whether it is possible to proceed towards a proof of the Black Hole Uniqueness Conjecture along the lines sketched above. We remark, however, that a related wave-type equation was employed in a uniqueness proof of Kerr for small angular momenta via completely different methods, see \cite{MR2672799,MR2461426}. 

\section{ALF instantons} \label{sec:ALF-4mfd}

In this section, we introduce the notions of ALF four-manifold and ALF instanton and prove some results on the asymptotics of ALF spaces that will be needed in the proof of the main results. We shall use the following conventions throughout the paper. We use abstract index notation following \cite{PR:I}. Unless otherwise stated, $(\cM, g_{ab})$ is assumed to be a smooth, complete, orientable Riemannian four-manifold, with Levi-Civita covariant derivative $\nabla_a$. The Riemann and Ricci tensors, and the scalar curvature are given by $R_{abc}{}^d \GenVec_d = (\nabla_a \nabla_b - \nabla_b \nabla_a) \GenVec_d$, $R_{ab} = R^c{}_{acb}$, $R =R^a{}_a$, respectively. We will sometimes use index-free notation, and for example write $g = g_{ab} dx^a dx^b$ for the metric tensor. The norm of a tensor $\varpi_{ab \cdots d}$ is defined by $|\varpi|^2 = \varpi_{ab \cdots d} \varpi^{ab \cdots d}$.  

\begin{definition}[\protect{\cite[Def. 1.1]{Biquard:Gauduchon}}] \label{def:BG} 
We say that a complete Riemannian four-manifold $(\cM, g)$ is \emph{asymptotically locally flat} (ALF) if the following holds. 
\begin{enumerate} 
\item \label{point:Mring} 
$\cM$ has an end diffeomorphic to $\cMring = (A,\infty) \times L$ where $L$ is $S^1 \times S^2$, $S^3$, or a finite quotient thereof.
\item \label{point:triple} 
There is a triple $(\eta, T, \gamma)$ defined on $L$, where $\eta$ is a 1-form, $T$ a vector field,  with 
\begin{align} 
i_T \eta = 1, \quad i_T d\eta = 0
\end{align} 
and $\gamma$ defines a $T$-invariant metric on the distribution $\ker \eta$. 
\item 
The transverse metric $\gamma$ has constant curvature $+1$.
\item 
Let 
\begin{align} \label{eq:gring}
\mathring{g} = dr^2 + r^2 \gamma + \eta^2
\end{align} 
be a metric on $(A,\infty) \times L$, with Levi-Civita covariant derivative $\mathring{\nabla}$. Here $\gamma$ is extended to a tensor on $L$ with $\ker \gamma$ generated by $T$.  We assume that the metric $g$ restricted to the end is of the form  
\begin{align} \label{eq:gringh} 
g = \mathring{g} + h
\end{align} 
where $h$ is a symmetric 2-tensor such that for all non-negative integers $k$, there are constants $C_k$ with   
\begin{align} 
|\mathring{\nabla}^k h|_{\mathring{g}} \leq C_k r^{-1-k} .
\end{align} 
\end{enumerate} 
An ALF manifold with $L = S^1 \times S^2$ is called \emph{asymptotically flat} (AF). If $(\cM, g_{ab})$ is ALF with model $\mathring{\cM} = (A, \infty) \times L$, we shall refer to $L$ as the boundary at infinity of $(\cM, g_{ab})$.  
\end{definition} 

\begin{definition}[ALF instanton]  
Let $(\cM, g_{ab})$ be a complete four-dimensional Ricci-flat Riemannian space. Then, $(\cM, g_{ab})$ is an ALF gravitational instanton if it is ALF in the sense of definition \ref{def:BG}. 
\end{definition} 

\begin{remark} \label{rem:T}
\begin{enumerate} 
\item In view of definition \ref{def:BG},
\begin{align} 
\Lie_T \eta = d(i_T \eta) + i_T d\eta = 0,
\qquad
\Lie_T \gamma = 0,
\end{align} 
and hence $T$ is a Killing field for $\gring$,
$\Lie_T \gring = 0$.
\item 
Definition \ref{def:BG} is compatible with the case when the orbits of $T$ do not close. This holds for generic AF models. 
\end{enumerate} 
\end{remark} 

The following remark highlights some features of the geometry of AF instantons. 
\begin{remark}[AF Instanton, see \protect{\cite[Definition 2.1]{MR4809333}}] \label{def:AF} 
\begin{enumerate} 
\item 
Let $\kappa, \Omega \in \Reals$, $\kappa \ne 0$, $|\Omega/\kappa| < 1$. Consider $\Reals^4$ with coordinates $(\tau, r, \theta, \phi) \in \Reals \times \Reals_+ \times [0,\pi] \times [0,2\pi]$ so that the flat metric takes the form 
\begin{align} 
\d\tau^2 + \d r^2 + r^2(\d\theta^2 + \sin^2\theta \d\phi^2) .
\end{align} 
Let $(\mathring{\cM}, \mathring{g}_{ab})$ be the flat space defined as $\Reals^4/\sim$ where the equivalence relation $\sim$ is given by the identification 
\begin{align}\label{eq:period} 
(\tau, r,\theta,\phi) \sim (\tau+2\pi/\kappa,r,\theta,\phi + 2\pi\Omega/\kappa) .
\end{align}
Introducing new Killing coordinates $\tilde\tau, \tilde\phi$ by 
\begin{align} 
\tau = \frac{1}{\kappa} \tilde \tau, \quad \phi = \frac{\Omega}{\kappa} \tilde \tau + \tilde \phi , 
\end{align} 
we have that the identification 
\begin{align} \label{eq:tperiod}
(\tilde\tau, \tilde\phi)\sim (\tilde \tau + 2\pi, \tilde \phi + 2\pi) 
\end{align} 
corresponds to \eqref{eq:period}. 
A non-flat ALF instanton with model $(\cM, \mathring{g}_{ab})$ is said to be AF, with parameters $\kappa, \Omega$.
\item If  
\begin{align} \label{eq:Omkap} 
\Omega/\kappa = q/p, \quad \text{with $q,p$ mutually prime integers.}
\end{align} 
then the action of $T = \partial_\tau$ has closed orbits, and $(\mathring{\cM}, \mathring{g}_{ab})$ is a model of an AF $S^1$-instanton. The generic orbits of $\partial_\tau$ have period $2\pi p/\kappa$ where $p,\kappa$ are as in \eqref{eq:Omkap}. At the poles $\theta = 0,\pi$, where $\partial_\phi$ vanishes, we have $\partial_\tau = \kappa \partial_{\tilde \tau}$ with the exceptional period $2\pi/\kappa$. 
\end{enumerate}
\end{remark} 

\begin{lemma}[\protect{\cite[Lemma 1.7]{Biquard:Gauduchon}}] \label{lem:H1zero} 
Let $(\cM, g_{ab})$ be an ALF instanton. Then $H^1_{\text{\rm dR}}(\cM) = 0$. 
\end{lemma} 

\begin{definition} \label{def:ordo} 
Let $(\cM, g_{ab})$ be an ALF four-manifold and let $r$ be as in Definition \ref{def:BG}. For a tensor field $t$ on $(\cM, g_{ab})$, we say that  $t =  O(r^\alpha)$ if  there is a constant $C$ such that $|t| \leq C r^\alpha$ for $r \geq A$, and write 
\begin{align} 
t = O^*(r^\alpha)
\end{align} 
if $|\nabla^k t| = O(r^{\alpha-k})$, for all non-negative integers $k$. 
\end{definition} 

\begin{remark} \label{rem:cone}   
Let $g$ be an ALF metric as in Definition \ref{def:AF}. As explained \cite[Proof~of~Prop.~1.6]{Biquard:Gauduchon},  $\eta$ has the local form $\eta = dt + \alpha$ where $dt(T) = 1$, and $\alpha$ is a $1$-form on the quotient of $L$ by the action of $T$, satisfying $\alpha = O^*(r^{-1})$. On sufficiently small domains in $L$ of the form $U \times (-t_0, t_0)$ with $U \subset S^2$, we have 
\begin{align} 
r^2 \gamma + \eta^2 = r^2 \sigma + dt^2 + O^*(r^{-1}) , 
\end{align} 
where $\sigma$ is the standard metric on $S^2$, and hence 
\begin{align} 
g = dr^2 +  r^2 \sigma + dt^2 + O^*(r^{-1}) = \delta + dt^2 + O^*(r^{-1})
\end{align} 
on 
\begin{align} 
C \times (-t_0,t_0) \subset \Reals^3 \times (-t_0,t_0). 
\end{align} 
Here $C$ is the cone 
\begin{align} 
C = (A, \infty) \times U \subset \Reals^+ \times S^2 \subset \Reals^3  
\end{align} 
and $\delta$ denotes the flat metric on $\Reals^3$. 
It follows that the curvature of $g$ satisfies $\Riem = O^*(r^{-3})$.

\end{remark} 

\begin{definition}\label{def:ALF-Ak} 
\begin{enumerate}
\item
An ALF instanton with $L$ an $S^1$ bundle over $S^2$ with Euler number $e$ is referred to as ALF-$A_k$, with $k = -e -1$. 
We refer to the particular case with $L = S^2 \times S^1$ as AF, or ALF-$A_{-1}$. 
\item 
An ALF instanton with $L$ a $S^1$ bundle over $\RP^2$ is said to be of type ALF-$D_k$. 
\end{enumerate}
\end{definition} 

\begin{remark} 
An ALF instanton of topology $S^4 \setminus S^1$ or $\CP^2\setminus S^1$ is ALF-$A_{-1}$, or equvalently AF. These cases correspond to the Kerr and Chen--Teo families of instantons, respectively. Similarly, an ALF instanton with topology for example $\CP^2 \setminus \{\text{\rm pt.}\}$, corresponding to the Taub-bolt instanton,
is ALF-$A_0$. 
\end{remark}  

\begin{definition}[ALF $S^1$-instanton]\label{def:S1inst}
Let $(\cM, g_{ab})$ be an ALF instanton. Then $(\cM, g_{ab})$ is an $S^1$-instanton if it admits an effective $S^1$ action by isometries, generated by a Killing field with uniformly bounded norm.
\end{definition}

\begin{remark}
The Euclidean Kerr instanton is AF with parameters $\kappa, \Omega$ corresponding to surface gravity and rotation speed, respectively. The non-rotating Euclidean Schwarzschild instanton is the limit of Kerr with $\Omega = 0$ and is $S^1$-symmetric. On the other hand, the Euclidean Kerr instanton is an $S^1$-instanton only if $\kappa/\Omega = q/p$ for mutually prime integers $q,p$. See \cite{1998PhRvD..59b4009H}.
\end{remark}

\begin{remark} \label{rem:seifert} 
Recall that $L$ in Definition~\ref{def:BG} has topology $S^2 \times S^1$, $S^3$ or a quotient thereof. We can refer to these as lens spaces, see \cite[\S 3]{MR3785783}.

Assuming that the orbits of $T$ are closed, the manifold $L$ in Definition~\ref{def:BG} is a Seifert fibration, that is a circle fibration over a base orbifold $O$.  Recall that by Proposition \ref{prop-asymptotics-killing}, if $(\cM, g_{ab})$ is an $S^1$-instanton, then  $T$ has closed orbits. 

The orbifold Euler characteristic is denoted by $\chi[O]$, see \cite[\ S 2]{MR705527} for background. If $O$ has $\tilde O$ as underlying surface, and has $m$ cone points with angles $2\pi/q_i$, $i=1,\dots, m$, then 
$\chi[O] =\chi[\tilde O] - \sum (1-1/q_i)$, see \cite[p. 427]{MR705527}. 

If $O$ is the base of the Seifert fibration of a lens space, then either  $O$ is orientable in which case it has $S^2$ as underlying surface,  with 2 cone points, $S^2(n,n)$, while if it is non-orientable it has $\RP^2$ as underlying surface, with one cone point, $\RP^2(n)$. The orbifold Euler characteristics are 
\begin{align} 
\chi[S^2(n,n)] ={}& 2 - 2(1-1/n) = 2/n , \\ 
\chi[\RP^2(n)] ={}& 1 - (1-1/n) = 1/n .
\end{align} 
For a manifold of type ALF-$A_k$ we have $O=S^2$ and $\chi[O] = 2$, while for a manifold of type ALF-$D_k$ we have $O=\RP^2$ and $\chi[O] = 1$. For a Seifert fibration of a lens space, we have $O = \RP^2(n)$ for $n \in \mathbb{N}$, with orbifold Euler characteristic $\chi[O] = 1 - (1 - 1/n) = 1/n$, see \cite[\ S 2]{MR705527}. For Seifert fibrations of $S^2 \times S^1$, see \cite[\S 4]{MR705527}. The Gauss--Bonnet theorem for orbifolds \cite[\ S 2]{MR705527} states that 
\begin{align} 
\int_O K d\mu = 2\pi \chi[O], 
\end{align}
where $K$ is the Gauss curvature. 
In the present case, we have $K=1$, which gives the following expressions for the area $A[O]$ of $O$, 
\begin{align} 
A[O] = 2\pi \chi[O] .
\end{align} 
\end{remark}

\section{The fixed point set} \label{sec:S1ALF}

Let $(\cM, g_{ab})$ be an orientable ALF  four-manifold with an effective isometric $S^1$-action generated by the Killing field $\xi^a$, and let $\Fixed$ be the fixed point set of the $S^1$ action.  The components of $\Fixed$ consist of isolated points called nuts, $\nut_1, \dots, \nut_{\nnuts}$ and two-surfaces called  bolts $\bolt_1, \dots \bolt_{\nbolts}$.  The fixed points set is orientable. Unless $\cM$  is flat, the fixed point set $\Fixed$ is non-empty \cite{MR3973911}. Hence, in view of the fact that the Euler characteristic of $\cM$ is even and that bolts are oriented and have even Euler characteristic,  in case $\cM$ is not flat, $\Fixed$ contains at least a bolt or two nuts. 

Let $\nut$ be a nut. Working in a suitable normal coordinate system, the skew-adjoint tensor $\nabla_a \xi^b$, 
evaluated at $\nut$ can be put in the canonical form  
\begin{align} \label{eq:nablaxi(nut)}
\nabla_a \xi^b \bigg{|}_{\nut} 
= 
\begin{pmatrix}
0&-\kappa^1&0&0 \\
\kappa^1&0&0&0 \\
0&0&0&-\kappa^2
\\0&0&\kappa^2&0
\end{pmatrix}
\end{align} 
and thus has non-zero skew eigenvalues, the surface gravities $\kappa^1, \kappa^2$. Since $\xi^a$ generates an $S^1$-action, $\kappa^1/\kappa^2$ must be rational, and hence there are positive, mutally prime integers $w_1, w_2$ called weights of $\nut$, and $\eps(\nut) = \pm 1$ called the orientation of $\nut$ such that    
\begin{align} 
\frac{\kappa^1(\nut)}{\kappa^2(\nut)} = \eps(\nut) \frac{w^1(\nut)}{w^2(\nut)} .
\end{align} 
A nut with $w^1=w^2=1$ is called self dual. Near $\nut$, the $S^1$-action has exceptional orbits with periods $2\pi/|\kappa^i(\nut)|$, $i=1,2$, while the generic orbit has period 
\begin{align} 
\frac{2\pi w^1(\nut)}{|\kappa^1(\nut)|} = \frac{2 \pi w^2(\nut)}{|\kappa^2(\nut)|}
\end{align}  
The fixed point data of $\nut$ is the triple $(\eps(\nut), w^1(\nut), w^2(\nut))$. When it is convenient we shall write simply $\eps=\pm$ instead of $\eps=\pm 1$. 

A bolt $\bolt$ is an orientable, totally geodesic two-surface in $(\cM, g_{ab})$, and thus has even Euler characteristic $\chi(\bolt)$. In addition, the Euler characteristic of the normal bundle of $\bolt$ is equal to the self-intersection number $\bolt\!\cdot\!\bolt$. For $p \in \bolt$, working in a suitable coordinate system, $\nabla_a \xi^b$ can be put in the canonical form 
\begin{align} \label{eq:nablaxi(bolt)}
\nabla_a \xi^b \bigg{|}_{p \in \bolt} 
= 
\begin{pmatrix}
0&0&0&0 \\
0 &0&0&0 \\
0&0&0&-\kappa \\
0&0&\kappa&0
\end{pmatrix}
\end{align} 
and thus has skew eigenvalues $0, \kappa$ for some surface gravity $\kappa \ne 0$. The orbits of the $S^1$-action near $\bolt$ all have period $2\pi/|\kappa|$. 

\begin{lemma}[Generic period] \label{lem:generic}
There is $G \in \Reals$ such that for any component $Z$ of $\Fixed$ we have    
\begin{align} 
G ={}& \frac{w^1(Z)}{|\kappa^1(Z)|} = \frac{w^2(Z)}{|\kappa^2(Z)|}, \quad \text{if $Z$ is a nut} \\
G ={}& \frac{1}{|\kappa(Z)|}, \quad \text{if $Z$ is a bolt}
\end{align} 
In particular, the generic period for the $S^1$-action is $2\pi G$. 
\end{lemma} 
\begin{proof} 
The exponential map based at any nut is surjective. The same statement holds for the normal exponential map at a bolt. It follows from this that all periods are mutually divisible. This proves the statement. 
\end{proof}

\subsection{Topological invariants} \label{sec:topin}

Let $(\cM, g_{ab})$ be an ALF $S^1$ instanton, then the Euler characteristic of $\cM$ can be computed in terms of the fixed point set as
\begin{align} \label{eq:chinutbolt}
\chi[\cM] = \nnuts + \sum_{i=1}^{\nbolts} \chi[\bolt_i], 
\end{align} 
see \cite{HirzebruchBergerJung}.

We also need to compute the signature of $\cM$ in terms of the fixed point set. Since this computation involves a contribution from infinity we will make the simplifying assumption that $(\cM, g_{ab})$ is an ALF$-A_k$ $S^1$ instanton. Then the manifold $L$ in Definition~\ref{def:BG} is the total space of an $S^1$-bundle over $S^2$ with Euler number $e = -k-1$. Let $N$ be the oriented total space of the corresponding disk bundle over $S^2$, so that $\partial N = L$. Denote $N$ with the opposite orientation by $\overline{N}$, then the manifolds $\cM$ and $\overline{N}$ can be glued along $L$ to form an oriented manifold $\cM \cup \overline{N}$. 

The $S^1$ action on $L$ extends to an $S^1$ action on $N$ which rotates the fibres and has as fixed point set the single bolt $B_N$ which is the zero section of the disk bundle and has self-intersection number $B_N \!\cdot\! B_N = e$. With the opposite orientation on $N$ we have $B_{\overline{N}} \! \cdot \! B_{\overline{N}} = -e$. The second homology group of $N$ is generated by $B_N$ and the intersection form is thus given by the $1 \times 1$ matrix $[B_N \! \cdot \! B_N] = [e]$. We conclude that $\sign[N] = \sgn(e)$, where $\sgn(e) = 0$ if $e=0$ and $\sgn(e) = e/|e|$ otherwise.

From Novikov additivity of the signature and the $G$-signature theorem \cite[Theorem~6.12]{AS-III}  and \cite[Proposition~6.18]{AS-III}, 
we get
\begin{align}
\sign[\cM] - \sign[N] 
&=\sign[\cM] + \sign[\overline{N}] \nonumber \\
&= \sign[\cM \cup \overline{N}] \nonumber \\
&= \sign[g,\cM \cup \overline{N}] \nonumber \\
&= - \sum_{i=1}^{\nnuts} \cot\frac{\kappa^1(\nut_i) t}{2}\cot\frac{\kappa^2(\nut_i) t}{2} \nonumber \\
&\qquad
+ \sum_{i=1}^{\nbolts} \csc^2 \frac{\kappa(\bolt_i) t}{2} B_i \! \cdot \! B_i \nonumber \\
&\qquad
+ \csc^2 \frac{\kappa(\bolt_N) t}{2} B_{\overline{N}} \!\cdot\! B_{\overline{N}},
\end{align}
where the notation we use is related to the notation in \cite[Proposition~6.18]{AS-III} by $\alpha_i = \kappa^1(\nut_i) t$, $\beta_i = \kappa^2(\nut_i) t$ for nuts and $\theta_i = \kappa(\bolt_i) t$ for bolts. Here $t$ must be chosen so that none of the angles is a  multiple of $\pi$ so that the action on the normal bundle does not contain a term $N^g(-1)$ in the notation of \cite{AS-III}, so that the proof of \cite[Proposition 6.18]{AS-III} applies. 
We thus get
\begin{align}
\sign[\cM] 
&= - \sum_{i=1}^{\nnuts} \cot\frac{\kappa^1(\nut_i) t}{2}\cot\frac{\kappa^2(\nut_i) t}{2} \nonumber \\
&\qquad
+ \sum_{i=1}^{\nbolts} \csc^2 \frac{\kappa(\bolt_i) t}{2} B_i \!\cdot\! B_i \nonumber \\
&\qquad
- \csc^2 \frac{\kappa(\bolt_N) t}{2} e  +\sgn(e).
\end{align}
Set 
\begin{align}\label{eq:g-def-G} 
g = e^{it/G}, 
\end{align} 
where $G$ is the generic period of Lemma~\ref{lem:generic}. Then
\begin{align}
\cot\frac{\kappa^j(\nut_i) t}{2} 
= i \sgn(\kappa^j(\nut_i)) \frac{g^{w^j(\nut_i)} + 1}{g^{w^j(\nut_i)} - 1},
\end{align}
so 
\begin{align}
\cot\frac{\kappa^1(\nut_i) t}{2} 
\cot\frac{\kappa^2(\nut_i) t}{2} 
= - \epsilon(\nut_i) 
\frac{g^{w^1(\nut_i)} + 1}{g^{w^1(\nut_i)} - 1} \cdot 
\frac{g^{w^2(\nut_i)} + 1}{g^{w^2(\nut_i)} - 1}.
\end{align}
Further, 
\begin{align}
\csc^2 \frac{\kappa(\bolt_i) t}{2} 
= -\frac{4g}{(g-1)^2}.
\end{align}
Note that Lemma~\ref{lem:generic} applies also to the bolt $B_N$. Together we have 
\begin{align} \label{eq:G-sign}
\sign[\cM] 
&= 
\sum_{i=1}^{\nnuts} \epsilon(\nut_i) 
\frac{g^{w^1(\nut_i)} + 1}{g^{w^1(\nut_i)} - 1} \cdot 
\frac{g^{w^2(\nut_i)} + 1}{g^{w^2(\nut_i)} - 1} \nonumber \\
&\qquad
+ \frac{4g}{(g-1)^2} \left( e - \sum_{i=1}^{\nbolts} B_i \!\cdot\! B_i \right) 
+ \sgn(e).
\end{align}
Importantly,  \eqref{eq:G-sign} holds as an algebraic identity with $g$ an independent variable, see \cite[\S 5.8]{HirzebruchBergerJung}. As remarked in the introduction, this fact plays an important role in \cite{2018JGP...133..181J}, see Section \ref{sec:fps} below. In particular we can set $g=0$ and obtain  
\begin{align} \label{eq:G-sign-g0}
\sign[\cM] = \sum_{i=1}^{\nnuts} \epsilon(\nut_i) + \sgn(e) .
\end{align} 

\begin{remark}
This argument to compute the signature also works for a general ALF manifold $\cM$. Then the manifold $N$ bounding the Seifert fibered manifold $L$ can be constructed using equivariant plumbing as in \cite[Chapter~2]{Orlik-seifert-manifolds}. The signature and the fixed point data of $N$ can then be computed from the corresponding plumbing graph. Since we do not need this for the results of the present paper we leave the details for future work.
\end{remark}

\subsection{Structure of the fixed point set} \label{sec:fps}
In \cite{2018JGP...133..181J}, Jang has provided a complete analysis of the fixed point set for $S^1$ actions on compact, oriented four-manifolds. The proofs however generalize to the situation under consideration here, namely an ALF $S^1$-instanton. In particular, near a bolt, the $S^1$ action is free, with the generic period $2\pi G$, see Lemma \ref{lem:generic}. Further, the identity 
\begin{align} 
\frac{4g}{(1-g)^2} = \frac{(1+g)^2}{(1-g)^2} - 1
\end{align} 
allows us to rewrite equation \eqref{eq:G-sign} as the $G$-signature formula for a compact, oriented four-manifold $\tilde{\cM}$ with an $S^1$-action that has nut struture corresponding to the nuts of $\cM$ with 
\begin{align} 
e-\sum_{i=1}^{\nbolts} \bolt_i \!\cdot\! \bolt_i
\end{align} 
additional (fictitious) nuts. The algebraic facts regarding the $G$-signature formula used in \cite{2018JGP...133..181J} therefore apply directly in our situation. 

Taking these facts into account, one may verify that the proofs of the following lemmas generalize to the case of ALF $S^1$-instantons. 

\begin{lemma}[Weight balance \protect{\cite[Lemma 2.6]{2018JGP...133..181J}}] \label{lem:2.6} Let $(\cM, g_{ab})$ be an ALF $S^1$ instanton, and assume the fixed point set $\Fixed$ of $\cM$ contains nuts. Let $w$ be a positive integer that occurs as the weight of a nut in $\Fixed$. Then the number of times $w$ occurs as a weight among all nuts of $\cM$, counted with multiplicity, is even. 
\end{lemma} 
We shall make use of the following notation. Let $\cM$ be a space with an $S^1$-action. For $w >1$, let $\cM^{\Zbb_w}$ be the fixed point set of the discrete subgroup $\Zbb_w \subset S^1$. By \cite[Lemma 2.2]{2018JGP...133..181J}, the proof of which applies in the current situation,  we have that if a component $Z \subset \cM^{\Zbb_w}$ contains a fixed point $\nut$ for $S^1/\Zbb_w$, then $Z = S^2$, and $Z$ contains exactly one more fixed point $Q \subset \Zbb_w$ which is also a nut. Further, one of the weights of $Q$ is $w$. 
\begin{lemma}[Companion nuts \protect{\cite[Lemma 3.5]{2018JGP...133..181J}}] \label{lem:3.5}
Let $(\cM, g_{ab})$ be an ALF $S^1$-instanton with fixed point set $\Fixed$ and let $\nut \in \Fixed$ be a nut with weights $a,w$, with $w > 1$. Then there is a unique nut $Q \in S^2 \subset \cM^{\Zbb_w}$. Let $b$ denote the other weight of $Q$. Then 
\begin{enumerate} 
\item 
If $\eps(Q) = \eps(\nut)$, then $a \equiv -b \mod w$
\item 
If $\eps(Q) = - \eps(\nut)$, then $a \equiv b \mod w$. 
\end{enumerate} 
\end{lemma} 

\begin{lemma}[Nuts with highest weights \protect{\cite[Lemma 3.6]{2018JGP...133..181J}}] 
\label{lem:3.6}
Let $(\cM, g_{ab})$ be an ALF $S^1$-instanton with fixed point set $\Fixed$ and let $\nut \in \Fixed$ be a nut with weights $a,w$. Assume that $w$ is the biggest of all weights, and that $w>1$. Then there is another nut $Q$ with weights $b,w$ and the following holds.  
\begin{enumerate} 
\item 
If $\eps(Q) = \eps(\nut)$, then $w=a+b$.
\item 
If $\eps(Q) = - \eps(\nut)$, then $a=b$.
\end{enumerate} 
\end{lemma} 

\begin{remark}[Examples of fixed point structures]
\label{rem:nutfacts} 
Let $(M, g_{ab})$ be a four-dim\-en\-sional, oriented AF $S^1$ space, with only isolated fixed points. 
\begin{enumerate} 
\item 
In order for the right hand side of \eqref{eq:G-sign} to be constant, it must hold that $\nnuts \geq 2$. See \cite[Remark 3.2]{MR3343967}. 
\item \label{point:2nut} 
If $\nnuts=2$, then \eqref{eq:G-sign} implies $\sign[\cM]=0$, and the nuts have weights 
\begin{align}
\{\pm, a, b\}, \quad  \{\mp,a,b\},
\end{align} 
see \cite[Remark 3.4]{MR3343967}. The Euclidean Kerr family provides an example with $n=2$ and $R_{ab} = 0$. Euclidean Kerr has Petrov type $D^+D^-$. 
\item \label{point:3nut} 
For $n=3$, the weights of the nuts are 
\begin{align} 
\{\pm, a,b\} , \quad \{\mp, a, a+b\}, \quad \{ \mp, b, a+b\}
\end{align} 
for some mutually prime, positive integers $a,b$, and $\sign[\cM] = \pm 1$, agreeing with the orientation of nut 1. See \cite[Theorem 1.8]{MR3343967}, \cite[Theorem 7.1]{2018JGP...133..181J}. The Chen--Teo example \cite{2011PhLB..703..359C} has $n=3$ and $R_{ab} = 0$, and is one-sided type $D$.
\item 
See \cite[Theorem 7.2]{2018JGP...133..181J} for the case of $n=4$.  
\end{enumerate} 
\end{remark}

\section{Geometry with a Killing field} \label{sec:S1inst} 

In this section, we derive some useful facts for Ricci-flat four-manifolds with a non-vanishing Killing field. The presentation here follows \cite{gibbons:hawking:1979} but adds more detail. 
Let $(\cM, g_{ab})$ be a Ricci-flat four-manifold and let 
$\xi^a$ be a Killing vector,
\begin{align}
\nabla_{a}\xi_{b} + \nabla_{b}\xi_{a}={}&0.
\end{align}
The scaling of $\xi^a$ is arbitrary throughout this section; it will be fixed via asymptotic conditions in Section \ref{sec:divident}. 
It is convenient to decompose certain fields into self dual (SD) and anti-self dual (ASD) parts, which we denote by superscripts $+$ and $-$, respectively.

\begin{definition}
Define the following covariant expressions derived from $\xi^{a}$,
\begin{subequations} \label{eq:concomitants} 
\begin{align}
\lambda  ={}& \xi_{a} \xi^{a}, \label{def:lambda} \\
F_{ab}={}&\nabla_{a}\xi_{b}, \label{def:F_ab} \\
\PM{\Fcal}_{ab}={}&F_{ab} \pm \tfrac{1}{2} \epsilon_{ab}{}^{cd} F_{cd}, \label{def:Fcal_ab}\\
\PM{\Fcal}^2={}&\PM{\Fcal}_{ab} \PM{\Fcal}^{ab}, \label{def:Fcal2} \\
\PM{\sigma}_{a}={}&2 \PM{\Fcal}_{ab} \xi^{b}, \label{def:sigmaa} \\
\mu ={}& \frac{1}{2}(\nabla_a\xi_b)(\nabla^a\xi^b), \qquad
\nu = \frac{1}{4}\epsilon_{abcd}(\nabla^a\xi^b)(\nabla^c\xi^d), \label{eq:munudef}
\end{align}
\end{subequations}
and define the SD/ASD Weyl tensor and its square via
\begin{align}
\PM{\Wcal}_{abcd} ={}& W_{abcd} \pm \tfrac{1}{2} \epsilon_{cd}{}^{ef} W_{abef}, \\
 \PM{\Wcal}^2 ={}& \PM{\Wcal}_{abcd}\PM{\Wcal}^{abcd} 
\end{align}
\end{definition} 
\begin{remark} 
The quantities $\mu, \nu$ given in \eqref{eq:munudef} are related to $\PM{\Fcal}^2$ via 
\begin{align}\label{eq:F2munu}
\PM{\Fcal}^2 = 4(\mu \pm \nu).
\end{align}
\end{remark}

The next proposition states a number of identities for the fields that have just been introduced. The Lorentzian signature version  is well-known, see \cite{Mars:2000:MR1779516} and references therein. The proof is deferred to Appendix \ref{sec:xi-id-proof}.

\begin{prop} \label{prop:xi-id} 
The fields defined above have the following algebraic and differential properties
\begin{subequations} \label{eq:xi-id}
\begin{align}
\PM{\Fcal}_{ab} \PM{\sigma}^{b}={}&- \tfrac{1}{2} \PM{\Fcal}^2\xi_{a}, \label{eq:Fsigma} \\
\PM{\sigma}_{a} \PM{\sigma}^{a}={}& \PM{\Fcal}^2 \lambda, \label{eq:sigmasq}\\
\nabla_{c}\PM{\Fcal}_{ab}={}&- \PM{\Wcal}_{abcd} \xi^{d}, \label{eq:nablaFcal} \\
\nabla_{c}\PM{\Fcal}^2={}&-2  \PM{\Wcal}_{cabd} \xi^{a}\PM{\Fcal}^{bd}, \label{eq:nablaFsq}\\
\nabla_{[a}\PM{\sigma}_{b]} ={}& 0, \label{eq:dsigmapm} \\
\nabla_{a}\PM{\sigma}^{a}={}&\PM{\Fcal}^2. \label{eq:divsigma} \\
\nabla^a \nabla_a \PM{\Fcal}^2 ={}& 
- \PM{\Wcal}_{abcd} \PM{\Fcal}^{ab} \PM{\Fcal}^{cd}
+ \frac{1}{2} \lambda  \PM{\Wcal}^2\label{eq:laplaceFsq}
\end{align}
\end{subequations}
\end{prop}

Due to \eqref{eq:dsigmapm} and Lemma \ref{lem:H1zero} we can introduce the Ernst potentials $\PM{\Ecal}$ via
\begin{align} \label{eq:Epm}
\nabla_{a}\PM{\Ecal}={}&\PM{\sigma}_{a},
\end{align}
which, by \eqref{eq:divsigma}, \eqref{eq:sigmasq}, solve the Ernst equations
\begin{align}\label{eq:DeltaEcal}
\lambda \nabla^a \nabla_a \PM{\Ecal} = \nabla_a \PM{\Ecal} \nabla^a \PM{\Ecal}.
\end{align}
With the twist 1-form
\begin{align}
\omega_{a}={}&\epsilon_{abcd} F^{cd} \xi^{b},
\end{align}
we have the decomposition
\begin{align} \label{eq:dEcal}
\PM{\sigma}_{a}={}& \nabla_{a}\lambda \pm \omega_{a}.
\end{align}
In view of \eqref{eq:dEcal}  and \eqref{eq:dsigmapm}, we have that $\omega_{a}$ is closed, 
\begin{align} 
\nabla_{[a} \omega_{b]} = 0. 
\end{align} 
and hence by Lemma \ref{lem:H1zero}, there is a  twist potential $\omega$  such that 
\begin{align} 
\nabla_a \omega = \omega_a. 
\end{align} 
The Ernst potentials $\PM{\cE}$ take the form
\begin{align} \label{eq:Ecal-lam-om} 
\PM{\Ecal} = \lambda \pm \omega,
\end{align}
and we also find  
\begin{align}
\lambda F_{ab} ={}&- \tfrac{1}{2} \epsilon_{abcd} \xi^{c} \omega^{d}
-\xi_{[a}\nabla_{b]}\lambda .
\end{align}

\subsection{Charges} 

Following \cite[\S 5]{gibbons:hawking:1979}, let 
\begin{subequations}\label{eq:currents}
\begin{align}
J_{T}{}^{a}={}&\frac{\M{\sigma}^{a} -  \P{\sigma}^{a}}{2 \lambda^2} = - \frac{\omega^a}{\lambda^2}, \label{eq:JT-def} \\
J_{D}{}^{a}={}&\frac{\P{\Ecal} \M{\sigma}^{a} + \M{\Ecal} \P{\sigma}^{a}}{2 \lambda^2} = \frac{-\omega \omega^a + \lambda \nabla^a\lambda}{\lambda^2},\\
J_{E}{}^{a}={}&\frac{\P{\Ecal}^2 \M{\sigma}^{a} -  \M{\Ecal}^2 \P{\sigma}^{a}}{2 \lambda^2} = \frac{-(\lambda^2 + \omega^2)\omega^a + 2\lambda\omega\nabla^a\lambda}{\lambda^2} .
\end{align}
\end{subequations}
With the assumptions and definitions above, the currents $J_{T}{}^{a}, J_{D}{}^{a}, J_{E}{}^{a}$, called Translation, Dilation and Ehlers currents, respectively, are conserved, that is $\nabla_a J_{T}{}^{a} = \nabla_a J_{D}{}^{a} = \nabla_a J_{E}{}^{a} = 0$. Hence, integrating the currents over closed 3-surfaces enclosing but not intersecting nuts and bolts defines quasilocally conserved charges.  

Next, we will compute the charges corresponding to $J_T{}^a$.

\begin{definition}
Let $K$ be a nut or a bolt in $\cM$. Let $U$ be an open set containing $K$ with smooth boundary $S$, so that neither $U$ nor $S$ intersects any other nut or bolt. Let $n^a$ be the inward pointing unit normal to $S$. The charge of $K$ is 
\begin{align} \label{def:N(K)}
N(K) 
= \frac{1}{8\pi} \int_S J^T _a n^a \, d\mu
= -\frac{1}{8\pi} \int_S \lambda^{-2} \omega_a n^a \, d\mu.
\end{align}
\end{definition} 

We now discuss the geometry of bolts in more detail. For a bolt $\bolt$ we denote the volume form of its tangent bundle by $\eps^\perp{}_{ab}$ and the volume form of its normal bundle by $\eps^\parallel{}_{ab}$. They are both two-forms on $T\cM$ defined at points in $\bolt$, and they are related by 
\begin{align} \label{eq:epsperpepspara} 
\eps_{abcd} \eps^\perp{}^{cd} = 2 \eps^\parallel{}_{ab},
\qquad
\eps_{abcd} \eps^\parallel{}^{cd} = 2 \eps^\perp{}_{ab}.
\end{align}
On $\bolt$ we have
\begin{align} \label{eq:nablaxiepsperp} 
F_{ab} = \nabla_a \xi_b = \kappa \eps^\perp{}_{ab}
\end{align}
where $\kappa$ is constant, and 
\begin{align} \label{eq:Fcalepsperpepspara} 
\PM{\Fcal}_{ab} = \kappa (\eps^\perp{}_{ab} \pm \eps^\parallel{}_{ab}).
\end{align}
Define the curvature scalars
\begin{align}
\cR^{\parallel\parallel}
&= \frac{1}{2} R_{abcd} \eps^\parallel{}^{ab}\eps^\parallel{}^{cd}, \\ 
\cR^{\parallel\perp}
&= \frac{1}{2} R_{abcd} \eps^\parallel{}^{ab}\eps^\perp{}^{cd}, \\
\cR^{\perp\perp} 
&= \frac{1}{2} R_{abcd} \eps^\perp{}^{ab}\eps^\perp{}^{cd}.
\end{align}
Since $\bolt$ is totally geodesic  we have 
\begin{align} \label{eq:Rparaparagauss}
\cR^{\parallel\parallel} = R^\bolt = 2\mathcal{K}^\bolt,
\end{align}
and since $\cM$ is Ricci-flat,
\begin{align} \label{eq:Rperpperpgauss}
\cR^{\perp\perp} = R^\bolt = 2\mathcal{K}^\bolt.
\end{align}

\begin{prop} \label{prop:charges}
For a nut $\nut$ with surface gravities $\kappa^1,\kappa^2$  we have 
\begin{align}\label{nutgrav} 
N(\nut) = \frac{\pi}{2\kappa^1\kappa^2}.
\end{align}
For a bolt $\bolt$ with surface gravity $\kappa$ we have
\begin{align}
N(\bolt) 
={}&
- \frac{1}{8 \kappa^2} 
\int_\bolt \cR^{\parallel\perp} d\mu \nonumber \\ 
={}& 
-\frac{\pi}{2\kappa^2} \bolt \!\cdot\! \bolt .
\end{align}
\end{prop}

\begin{proof}
We will compute the integral \eqref{def:N(K)} when $S = S(\rho)$ is a connected component of the level set $\{\lambda=\rho\}$. Then $n^a = - \frac{\nabla^a \lambda}{|\nabla^a \lambda|}$. Since $\lambda^{-2} \omega_a$ is divergence free we have
\begin{align}
0 &= \nabla^a (\lambda^{-2} \omega_a)  \nonumber \\
&= -2 \lambda^{-3} (\nabla^a \lambda) \omega_a + \lambda^{-2} \nabla^a \omega_a \nonumber \\
&= 2 \lambda^{-3} |\nabla \lambda|  n^a \omega_a + \lambda^{-2} \nabla^a \omega_a,  
\end{align}
so
\begin{align}
\lambda^{-2} n^a \omega_a 
= - \frac{1}{2\lambda|\nabla\lambda| } \nabla^a \omega_a
= - \frac{1}{|\nabla\lambda^2| } \nabla^a \omega_a.
\end{align}
Here,   
\begin{align}
\nabla^a \omega_a
&= \nabla^a(\epsilon_{abcd} \xi^b (\nabla^c \xi^d) ) \nonumber \\ 
&= 
\epsilon_{abcd} (\nabla^a\xi^b) (\nabla^c \xi^d)
+ \epsilon_{abcd} \xi^b (\nabla^a \nabla^c \xi^d) \nonumber \\
&= 
4\nu
+ \epsilon_{abcd} \xi^b R^{dca}{}_f \xi^f \nonumber \\
&= 4\nu
\end{align}
where the second term vanishes by the Bianchi identity.

Let $\alpha \in \Reals$, and let $A$, $B$ be two expressions which satisfy $A = O(\lambda^{\alpha})$, $B = O(\lambda^{\alpha})$ near a component of the fixed point set. We write $A \simeq B$ if $A - B = O(\lambda^{\alpha + \delta})$ for some $\delta > 0$. 

Near a nut $\nut$ we have
\begin{align}
\nu \simeq 2 \kappa^1 \kappa^2
\end{align}
and in normal coordinates,
\begin{align}
\lambda \simeq 
(\kappa^1)^2 \left( (y^1)^2 + (y^2)^2 \right)
+ (\kappa^2)^2 \left( (y^3)^2 + (y^4)^2 \right).
\end{align}
Let $D(\delta)$ be the connected set containing $\nut$ where $\lambda \leq \delta$. For $\delta$ small, the coarea formula applied to the sets $S(\rho^{1/2})$ where $\lambda^2 = \rho$ tells us that
\begin{align}
\delta \cdot N(\nut)
&= \int_0^\delta N(\nut) \, d\rho \nonumber \\
&= -\frac{1}{8\pi} \int_0^\delta 
\int_{S(\rho^{1/2})} \lambda^{-2} \omega_a n^a \, d\mu 
\, d\rho \nonumber \\
&= \frac{1}{8\pi} \int_0^\delta 
\int_{S(\rho^{1/2})} \frac{4\nu}{|\nabla\lambda^2| } \, d\mu 
\, d\rho \nonumber \\
&\simeq 
\frac{8 \kappa^1 \kappa^2}{8\pi} \int_0^\delta 
\int_{S(\rho^{1/2})} \frac{1}{|\nabla \lambda^2|} \, d\mu
\, d\rho \nonumber \\
&= 
\frac{\kappa^1 \kappa^2}{\pi} \int_{D(\delta^{1/2})} 1 \, d\mu \nonumber \\
&= 
\frac{\kappa^1 \kappa^2}{\pi} \operatorname{Vol}(D(\delta^{1/2})).
\end{align} 
With a linear change of coordinates $(x^1,x^2,x^3,x^4) = (\kappa^1 y^1,\kappa^1 y^2,\kappa^2 y^3,\kappa^2 y^4)$ with determinant $(\kappa^1 \kappa^2)^2$ we have that $\lambda \leq \delta^{1/2}$ corresponds to $|x| \leq \delta^{1/4}$ and we find 
\begin{align}
\operatorname{Vol}(D(\delta^{1/2})) 
= \frac{1}{(\kappa^1 \kappa^2)^2} \cdot \frac{\pi^2(\delta^{1/4})^4}{2}
= \frac{\delta \pi^2}{2(\kappa^1 \kappa^2)^2}.
\end{align}
Together we have 
\begin{align}
N(\nut) = \frac{\pi}{2\kappa^1 \kappa^2}.
\end{align}
Next, working near a bolt $\bolt$, introduce local coordinates $x^a$ on $\cB$ and coordinates $y^a$ in the normal direction, with $|y(p)| \cong d(p,\cB)$. Using the identity 
\begin{align}
\eps^{\perp}{}_h{}^{[f} y^{g]} y^h = \frac{1}{2} |y|^2 \eps^{\perp fg},
\end{align}
we compute the Taylor expansion of $\nabla^a \xi^b$ to second order in the normal directions to $\bolt$,
\begin{align}
\nabla^a \xi^b 
&\simeq 
(\nabla^a \xi^b)|_\bolt + (\nabla_g \nabla^a \xi^b)|_\bolt y^g
+ \frac{1}{2} (\nabla_h \nabla_g \nabla^a \xi^b)|_\bolt y^g y^h \nonumber \\
&=
\kappa \eps^{\perp ab} + (R^{ba}{}_{gf} \xi^f)|_\bolt y^g
+ \frac{1}{2} (\nabla_h (R^{ba}{}_{gf} \xi^f))|_\bolt y^g y^h \nonumber \\
&= \kappa \eps^{\perp ab} +
\frac{1}{2}R^{ba}{}_{gf} (\nabla_h \xi^f)|_\bolt y^g y^h \nonumber \\
&=
\kappa \eps^{\perp ab} +
\frac{1}{2}\kappa R^{ba}{}_{gf} \eps^{\perp}{}_h{}^f y^g y^h \nonumber \\
&=
\kappa \left(
\eps^{\perp ab} + \frac{1}{4}|y|^2 R^{ba}{}_{gf}\eps^{\perp fg} 
\right).
\end{align}
We find that 
\begin{align}
4\nu 
&\simeq
\kappa^2 \eps_{abcd}
\left(
\eps^{\perp ab} + \frac{1}{4}|y|^2 R^{ba}{}_{gf}\eps^{\perp fg} 
\right)
\left(
\eps^{\perp cd} + \frac{1}{4}|y|^2 R^{dc}{}_{hi}\eps^{\perp ih} 
\right) \nonumber \\
&\simeq
\kappa^2 \eps_{abcd} \eps^{\perp ab} \eps^{\perp cd} 
+ \frac{1}{2}\kappa^2 |y|^2 \eps_{abcd} \eps^{\perp ab}
R^{dc}{}_{hi} \eps^{\perp ih} \nonumber \\
&=
0 + 
\kappa^2 |y|^2 \eps^{\parallel}{}_{cd} R^{dc}{}_{hi} \eps^{\perp ih} \nonumber \\
&=
- 2 \kappa^2 |y|^2 \cR^{\parallel\perp}. \label{eq:Rparper}
\end{align}
Near the bolt $\bolt$ we further have 
\begin{align}
\lambda \simeq \kappa^2 |y|^2, \qquad
\nabla^a \lambda \simeq 2 \kappa^2 y^a, \qquad 
|\nabla^a \lambda| \simeq 2 \kappa^2 |y|,
\end{align}
so
\begin{align}
\lambda^{-2} n^a \omega_a 
&= - \frac{1}{2\lambda|\nabla^a \lambda| } 4 \nu \nonumber \\
&\simeq
\frac{1}{2 \kappa^2 |y|^2 \cdot 2 \kappa^2 |y|} 2 \kappa^2 |y|^2 \cR^{\parallel\perp} \nonumber \\
&=
\frac{1}{2\kappa^2|y|}\cR^{\parallel\perp}. 
\end{align}
On the surface $S$ around $\bolt$ we have $\eps = \lambda \simeq \kappa^2 |y|^2$, so the surface $S$ consists of circles of circumference $\simeq 2\pi |y| \simeq 2\pi |\kappa|^{-1}\eps^{1/2}$ normal to $\bolt$. We thus find that
\begin{align}
N(\bolt) 
&= -\frac{1}{8\pi} \int_S \lambda^{-2} \omega_a n^a \, d\mu \nonumber \\
&\simeq
- \frac{1}{8\pi} 
\int_S \frac{1}{2\kappa^2|y|}\cR^{\parallel\perp} \, d\mu \nonumber \\
&\simeq
- \frac{1}{8\pi} 
\int_\bolt 2\pi |\kappa|^{-1}\eps^{1/2}
\frac{1}{2|\kappa| \eps^{1/2}}\cR^{\parallel\perp} \, d\mu \nonumber \\
&=
- \frac{1}{8 \kappa^2} 
\int_\bolt \cR^{\parallel\perp} \, d\mu.
\end{align}
Next, $\mathcal{R}^{\parallel\perp} d\mu = 4\pi e(NB)$, where $e(NB)$ is the Euler class of the normal bundle $NB$ of $B$ in $\cM$. From \cite[Theorem~11.17]{BottTu} we have
\begin{align}
\int_\bolt\mathcal{R}^{\parallel\perp} \, d\mu
= 4\pi \int_\bolt e(NB)
= 4\pi \bolt\!\cdot\! \bolt ,
\end{align} 
which gives the result. 
\end{proof}

\section{Divergence identities} \label{sec:divident}

In this section we derive a family of divergence identities, see  \eqref{eq:DivPsialpha} below, depending on a parameter $\beta$ taking real values. An alternative derivation on the orbit space, following \cite{MR747036,simon:1995}, will be given in Appendix~\ref{sec:quotient}. Equation \eqref{eq:DivPsialpha} with $\beta =1$ is the key ingredient of the uniqueness arguments presented in Section~\ref{sec:bound}.

After a suitable normalization of $\xi$, we may without loss of generality assume that the supremum of $\lambda$ is $1$. 
\begin{lemma}
Assume that the supremum of $\lambda$ is $1$. The following holds. 
\label{lem:Eprop}  
\begin{enumerate} 
\item \label{point:lmnO*} 
\begin{subequations}\label{eq:lomO*}
\begin{align} 
\lambda ={}& 1 + O^*(r^{-1}) \label{eq:lambdaO*}  \\
\intertext{The twist potential $\omega$ can be normalized so that $\omega \rightarrow 0$ at $\infty$ and in this case,} 
\omega ={}& O^*(r^{-1}) \label{eq:omegaO*} 
\end{align} 
\end{subequations} 
\item \label{point:cE-expand}
There are numbers $\PM{b} \leq 0$ and  $\delta > 0$ so that
\begin{align} \label{eq:cE-expand} 
\PM{\cE} = 1 + \PM{b} r^{-1} + O^*(r^{-1- \delta})
\end{align} 
If $\PM{\cE}$ is non-constant, then $\PM{b} < 0$. 
\item \label{point:cE} 
$\PM{\cE} \equiv 1$ or $-1 < \PM{\cE} < 1$ 
\end{enumerate} 
\end{lemma} 
\begin{proof} 

Since  $g_{ab}$ is Ricci-flat we have $\Delta \xi = 0$, and hence
\begin{align} 
\Delta |\xi|^2 = 2 \la \nabla\xi, \nabla\xi\ra . 
\end{align} 
By the maximum principle, we have $|\xi| < 1$. By Proposition \ref{prop-asymptotics-killing} it follows that $|\xi| \to 1$ and hence also $\lambda \to 1$ at infinity. 
We also have by Proposition \ref{prop-asymptotics-killing} that \begin{align} \label{eq:nablaxiO*2}
\nabla_a \xi_b = O^*(r^{-2})
\end{align} 
By Lemma \ref{lemma-piotr-limit} we may normalize $\omega$ to have $\omega \rightarrow 0$ at $\infty$. 
Equations \eqref{eq:lomO*} now follow from the definitions of $\lambda$ and $\omega$, Equation \eqref{eq:nablaxiO*2}, and Lemma \ref{lemma-piotr-limit}. 
From the definition of $\PM{\cE}$ and \eqref{eq:nablaxiO*2}, we have that $\PM{\cE} = 1 + O^*(r^{-1})$, and hence, by \eqref{eq:Epm}, \eqref{eq:divsigma}, and \eqref{eq:lambdaO*}, we have 
\begin{align} \label{eq:Duf}
\Delta \PM{\cE} = \PM{\Fcal}^2, \quad \text{where $\PM{\Fcal}^2 = O^*(r^{-4})$.}
\end{align}
By Proposition \ref{prop:poisson-g}, it follows that $\PM{\cE}$ has an expansion of the form \eqref{eq:cE-expand}. For $R$ large, define the set $U_R = \{r \leq R \}$ with boundary 
$\partial U_r = S_R = \{r = R \}$. From Gauss's theorem we get 
\begin{align}
\int_{\cM} \PM{\Fcal}^2 d\mu 
&= \int_{\cM} \Delta \PM{\cE} d\mu \nonumber \\
&= \lim_{R \to \infty} \int_{U_R} \Delta \PM{\cE} d\mu \nonumber \\
&= \lim_{R \to \infty} \int_{S_R} n^a \nabla_{a} \PM{\cE} d\mu \nonumber \\
&= \lim_{R \to \infty} \int_{S_R} \left( - \PM{b} r^{-2} + O^*(r^{-2- \delta})\right) d\mu \nonumber \\
&= - C \PM{b},
\end{align}
where $n^a$ is the outward pointing normal to $S_r$. Further, $C > 0$ is a constant related to the volume of $L$ in Definition~\ref{def:BG}, see similar computations in the proof of Lemma~ \ref{lem:bdrinfty}. It follows that $\PM{b} < 0$ unless 
$\PM{\Fcal}^2 \equiv 0$ in which case $\PM{\cE} \equiv 1$. This proves point \eqref{point:cE-expand}. If $\PM{b} < 0$, the maximum principle implies that $\PM{\cE} < 1$. We thus get $1 > \PM{\cE} = - \MP{\cE} + 2\lambda \geq - \MP{\cE} > -1$. This proves point \eqref{point:cE}.
\end{proof} 

With the normalization used in Lemma \ref{lem:Eprop}, we have 
\begin{subequations} \label{Elamom-infty} 
\begin{align} 
\lambda \to{}& 1 \quad \text{at $\infty$}, \\
\omega \to{}& 0 \quad \text{at $\infty$}, \\
\PM{\cE} \to{}& 1 \quad \text{at $\infty$}. 
\label{eq:Epm-infty} 
\end{align} 
\end{subequations} 
These condition will be imposed throughout the rest of the paper.

We also note the following equivalent conditions for $(\cM, g_{ab})$ to be half-flat, i.e. $\P{\Wcal}_{abcd} \equiv 0$ or $\M{\Wcal}_{abcd} \equiv 0$.
\begin{lemma} \label{prop:SD-characterizations}
For any sign $+$ or $-$, the following are equivalent. 
\begin{enumerate} 
\item \label{point:SD-1} 
$(\cM, g_{ab})$ is half-flat, that is $\PM{\Wcal}_{abcd} \equiv 0$,
\item \label{point:SD-2} 
$\PM{\Fcal} \equiv 0$, 
\item \label{point:SD-3} 
$\PM{\cE} \equiv 1$.
\end{enumerate}
\end{lemma}

\begin{proof}
If $\PM{\Wcal}_{abcd} \equiv 0$ it follows from \eqref{eq:nablaFsq} that $\nabla_{c}\PM{\Fcal}^2 \equiv 0$, that is $\PM{\Fcal}^2$ is constant. By Proposition \ref{prop-asymptotics-killing} we have that $\PM{\Fcal}^2 \to 0$ at infinity, and hence $\PM{\Fcal}^2 \equiv 0$.
On the other hand, if $\PM{\Fcal}^2 \equiv 0$ we have $\PM{\Fcal}_{ab} \equiv 0$ and from \eqref{eq:laplaceFsq} it follows that $\PM{\Wcal}_{abcd} \equiv 0$. This shows \eqref{point:SD-1} $\Leftrightarrow$ \eqref{point:SD-2}. 
From \eqref{eq:sigmasq} it follows that $\PM{\Fcal} \equiv 0$ if and only if $\PM{\sigma}_{a} \equiv 0$, which by \eqref{eq:Epm} is equivalent to  $\nabla_{a}\PM{\Ecal} \equiv 0$. Since $\PM{\Ecal} \to 1$ at $\infty$ by \eqref{eq:Epm-infty} , this means that $\PM{\Ecal} \equiv 1$. Hence \eqref{point:SD-2} $\Leftrightarrow$ \eqref{point:SD-3}.  
\end{proof}

\begin{definition}\label{def:mst}
On domains where $\PM{\Ecal} \neq 1$ we introduce the objects
\begin{subequations}
\begin{align}
\PM{\Scal}_{abcd}={}&\PM{\Wcal}_{abcd} - \frac{6 (\PM{\Fcal}_{ab} \PM{\Fcal}_{cd} -  \tfrac{1}{3} \PM{\Fcal}^2 \PM{\Ical}_{abcd})}{1 -  \PM{\Ecal}},  \label{mst}\\
\PM{\Scal}^2={}&\PM{\Scal}_{abcd} \PM{\Scal}^{abcd}, \\
\PM{\sbold}^2={}&\frac{\PM{\Fcal}^2}{(1 -  \PM{\Ecal})^4}, \label{eq:MSscalar}
\end{align}
\end{subequations}
respectively. Here $\PM{\Ical}_{abcd}= \tfrac{1}{4} ( g_{ac} g_{bd} -  g_{ad} g_{bc} \pm \epsilon_{abcd} )$ is the metric on SD and ASD 2-forms.
\end{definition}

\subsection{Characterizations of special geometry} 
The following result holds independently for the SD and ASD sides. Recall the Petrov classification of Weyl tensors on Riemannian 4-manifolds \cite{karlhede:1986}, see also \cite{PR:II,MR4809333}.
\begin{lemma} \label{lem:typeDchar}
$\PM{\Wcal}_{abcd} $ is of Petrov type $\PM{D}$ if one of the following equivalent conditions hold.
\begin{enumerate}
\item $\PM{\Scal}_{abcd} = 0$ and $\PM{\Wcal}_{abcd} \neq 0$.
\item $\PM{\sbold}^2 = \text{const.} \neq 0$.
\end{enumerate}
\end{lemma}
\begin{proof}
We first show the equivalence of the conditions. From \eqref{eq:nablaMSscalar} and $\PM{\Fcal}^2 \neq0$ we have (1) $\Rightarrow$ (2). From \eqref{eq:LaplaceMSscalarprel} we have (2) $\Rightarrow \PM{\Scal}^2 = 0$ which implies (1).

Suppose condition (1) holds. Then we have
\begin{align}
\PM{\Wcal}_{abcd} = \frac{6 (\PM{\Fcal}_{ab} \PM{\Fcal}_{cd} -  \tfrac{1}{3} \PM{\Fcal}^2 \PM{\Ical}_{abcd})}{1 -  \PM{\Ecal}} \neq 0, 
\end{align}
which implies that $\PM{\Wcal}_{abcd}$ has degenerate eigenvalues and hence is of Petrov type $\PM{D}$.
\end{proof}

\begin{prop} \label{prop:MS} 
For  $\beta \in \Reals$, the scalar $\PM{\sbold}$ satisfies the modified Laplace equation
\begin{align} \label{eq:LaplaceMSscalar}
(\nabla^{a} - \PM{\Gamma}^{a}) \nabla_{a} |\PM{\sbold}|^\beta ={}& \PM{V}(\beta) |\PM{\sbold}|^\beta,
\end{align}
with  
\begin{align}
\PM{\Gamma}_{a}={}&\frac{(1 + \MP{\Ecal}) }{(1 -  \PM{\Ecal}) \lambda} \PM{\sigma}_{a}, \label{eq:Gammapm}\\
\PM{V}(\beta) ={}& \frac{\beta \lambda \bigl(\PM{\Fcal}^2 \PM{\Scal}^2 + (\beta - 2) (\PM{\Fcal}\PM{\Scal})^2 \bigr) }{4 (\PM{\Fcal}^2)^2}, \label{eq:Vpm}
\end{align}
and where $(\PM{\Fcal}\PM{\Scal})^2 = \PM{\Fcal}^{ab} \PM{\Scal}^{ef}{}_{ab} \PM{\Fcal}^{cd} \PM{\Scal}_{efcd} $.
\end{prop}

\begin{proof}
The gradient of \eqref{eq:MSscalar}, using \eqref{eq:nablaFsq} and \eqref{eq:Epm}, is equal to
\begin{align} \label{eq:nablaMSscalar}
\nabla_{a}\PM{\sbold}^2 ={}& - \frac{2 \PM{\Fcal}^{cd} \PM{\Scal}_{abcd} \xi^{b}}{(1 - \PM{\Ecal})^4}.
\end{align}
 By the Maxwell-type equations \eqref{eq:nablaFcal} and the Bianchi identity
applied to \eqref{mst}, the contracted second derivative is given by 
\begin{align}\label{eq:MSid}
\nabla_{a}\nabla^{a}\PM{\sbold}^2={}& \frac{ \PM{\Scal}^2 \lambda}{2 (1 - \PM{\Ecal})^4} - \frac{ (1 +  \MP{\Ecal})}{ (1 - \PM{\Ecal})^5} \PM{\Fcal}^{ab} \PM{\Fcal}^{cd} \PM{\Scal}_{abcd}.
\end{align}
We notice that the last term can be absorbed by a lower order term involving \eqref{eq:Gammapm}, such that
\begin{align} \label{eq:LaplaceMSscalarprel}
(\nabla^{a} - \PM{\Gamma}^{a}) \nabla_{a}\PM{\sbold}^2 ={}&\frac{\PM{\Scal}^2 \lambda}{2 \PM{\Fcal}^2} \PM{\sbold}^2.
\end{align}
Now \eqref{eq:LaplaceMSscalar} follows by expanding in $\beta$ as follows,
\begin{align} 
(\nabla^{a} - \PM{\Gamma}^{a}) \nabla_{a}|\PM{\sbold}|^\beta 
={}& 
\tfrac{\beta}{2} (\PM{\sbold}^2)^{\beta/2-1} \left((\nabla^{a} - \PM{\Gamma}^{a}) \nabla_{a}\PM{\sbold}^2 + (\tfrac{\beta}{2} - 1)   (\nabla\PM{\sbold}^2 )^2 / \PM{\sbold}^{2}\right) \nonumber \\
={}& 
\frac{\beta \PM{\Scal}^2 \lambda}{4 \PM{\Fcal}^2} |\PM{\sbold}|^\beta + \tfrac{\beta(\beta-2)}{4} |\PM{\sbold}|^{\beta-4} \frac{4 \PM{\Fcal}^{cd} \PM{\Scal}_{abcd} \xi^{b}\PM{\Fcal}^{ef} \PM{\Scal}^a{}_{hef} \xi^{h}}{(1 - \PM{\Ecal})^8} \nonumber \\
={}& 
\frac{\beta \PM{\Scal}^2 \lambda}{4 \PM{\Fcal}^2} |\PM{\sbold}|^\beta + 
\frac{\beta(\beta-2)}{4} |\PM{\sbold}|^{\beta} \frac{ \PM{\Fcal}^{cd} \PM{\Scal}_{abcd} \PM{\Fcal}^{ef} \PM{\Scal}^{ab}{}_{ef} \lambda}{(\PM{\Fcal}^2)^2} .
\end{align}
Here we used \eqref{eq:LaplaceMSscalarprel} and \eqref{eq:nablaMSscalar} for the second equality. The third equality uses \eqref{eq:MSscalar} and
$ (\PM{\Scal}_{ibcd}\PM{\Fcal}^{cd}) (\PM{\Scal}^i{}_{hef}\PM{\Fcal}^{ef}) = \frac{1}{4} (\PM{\Scal}_{ijcd}\PM{\Fcal}^{cd} )(  \PM{\Scal}^{ij}{}_{ef}\PM{\Fcal}^{ef}) g_{bh}$ which is analogous to \eqref{eq:FcalFcal1contr} for two SD/ASD two-forms with one index contraction.
\end{proof}

\begin{remark} 
The quantities defined in Definition~\ref{def:mst} are derived from the Weyl tensor in a manner analogous to the complex tensors used in characterizations of the Lorentzian Kerr geometry by Mars \cite{MR1701088}. Identities analogous to \eqref{eq:MSid} as well as wave equations for $\PM{\Scal}_{abcd}$ were applied to the black hole uniqueness problem in \cite{MR2672799, MR2461426}. 
\end{remark} 

\begin{prop}\label{prop:V(beta)} 
For $\beta \geq 1/2$, the potentials $\PM{V}(\beta)$ given in \eqref{eq:Vpm} satisfy
\begin{align}
\PM{V}(\beta) \geq 0.
\end{align}
\end{prop}
\begin{proof}
For $\beta \geq 2$, \eqref{eq:Vpm} is a sum of squares and hence non-negative. For  $\beta \in \mathbb{R}$ we introduce the
quantities 
\begin{align}
\label{stst}
\PM{\mathcal{P}}_{abc} ={}&  \gamma_{a[b}\PM{\Wcal}_{c]efk} \xi^{e} \PM{\Fcal}^{fk}  + 4 \xi^{e} \xi^{f} \PM{\Wcal}_{eaf[b}  \PM{\sigma}_{c]},
\end{align}
where $ \gamma_{ab}= \lambda g_{ab} -  \xi_{a} \xi_{b}$. 
The tensors $\PM{\mathcal{P}}_{abc}$ introduced in \eqref{stst} are  the  Riemannian counterparts of the (complex)  ``spacetime Simon tensor'', see \cite[Definition 1]{MR1701088}. 
Their squares decompose into 
\begin{align}
2 \PM{\mathcal{P}}^2 \lambda^{-3} = \PM{\Fcal}^2 \PM{\Scal}^2 -\frac{3}{2} (\PM{\Fcal}\PM{\Scal})^2, 
\end{align}
which shows that \eqref{eq:Vpm}, is a sum of squares for $\beta \geq 1/2$.
\end{proof}

\begin{lemma}[Divergence Identity]
Let $J_{T}{}^{a}, J_{D}{}^{a}, J_{E}{}^{a}$ be as in \eqref{eq:currents}, and let  
\begin{align} \label{eq:Jcurrdef}
\PM{J}^{a}={}&\pm J_{T}{}^{a}  - 2 J_{D}{}^{a} \pm J_{E}{}^{a}
 = \frac{\pm (1 \mp \P{\Ecal})^2 \M{\sigma}^{a} \mp  (1 \pm \M{\Ecal})^2 \P{\sigma}^{a}}{2 \lambda^2}.
\end{align} 
For $\beta \in \Reals$,  the vector fields
\begin{align}
\PM{\Psi}^{(\beta)a}={}& \frac{1}{2} \PM{J}^{a} |\PM{\sbold}|^\beta
+ \frac{(1+\MP{\Ecal})(1-\PM{\Ecal})}{2\lambda}\nabla^{a}|\PM{\sbold}|^\beta \label{eq:PMPsia}
\end{align}
satisfy
\begin{align} \label{eq:DivPsialpha}
\nabla_a \PM{\Psi}^{(\beta)a} = \frac{(1 + \MP{\Ecal}) (1 -  \PM{\Ecal})}{2\lambda} \PM{V}(\beta) |\PM{\sbold}|^\beta.
\end{align}
\end{lemma}
\begin{proof}
To obtain  \eqref{eq:DivPsialpha}, recall that the
current $\PM{J}_a$ is conserved (cf. the remark after \eqref{eq:currents}) and that the gradient
of $\PM{\sbold}$ is given by \red{\eqref{eq:nablaMSscalar}}. The divergence of the second term on the right hand side of 
\eqref{eq:PMPsia} follows from 
\begin{align}
\left( \nabla_a + \PM{\Gamma}_a \right) \left( \frac{(1+\MP{\Ecal})(1-\PM{\Ecal})}{\lambda}\right) = - \PM{J}_{a}  \label{eq:nabla+Gamma}
\end{align}
and \eqref{eq:LaplaceMSscalar}.
\end{proof}
\begin{remark}
An alternative proof of \eqref{eq:DivPsialpha} is given by Lemma \ref{lem:divid}. See also point \eqref{point:divrem:1} of Remark \ref{divrem} .
\end{remark}

Let $\PM{\Psi}^{(\beta)}_a$ be as in \eqref{eq:DivPsialpha}. We now restrict to the case $\beta=1$, see Remark~\ref{rem:beta}. Set  
\begin{align} 
\PM{\Psi}_a  = \PM{\Psi}^{(1)}_a  , \quad \PM{V} = \PM{V}(1) ,
\end{align}
so that
\begin{align} \label{eq:Psi-def} 
\PM{\Psi}_a
= \frac{1}{2}\PM{J}_a \PM{\sbold}
+ \frac{(1+\MP{\Ecal})(1-\PM{\Ecal})}{2\lambda}\nabla_{a} \PM{\sbold} .
\end{align}

\begin{lemma} \label{lem:magic}  
Let $(\cM, g_{ab})$ be an ALF $S^1$-instanton. Then $\PM{\Psi}_a$ satisfies
\begin{align} \label{eq:divmagic} 
\nabla^a \PM{\Psi}_a = \frac{(1 + \MP{\Ecal}) (1 -  \PM{\Ecal})}{2\lambda} \PM{V} |\PM{\sbold}| ,
\end{align}
where $\PM{\sbold}$ is given by \eqref{eq:MSscalar}, and 
\begin{align} \label{eq:Vpm1}
\PM{V} ={}& \frac{\lambda \bigl(\PM{\Fcal}^2 \PM{\Scal}^2 - (\PM{\Fcal}\PM{\Scal})^2 \bigr) }{4 (\PM{\Fcal}^2)^2}, 
\end{align} 
In particular, 
\begin{align} 
\nabla^a \PM{\Psi}_a \geq 0
\end{align} 
with equality only if $\PM{\Wcal}_{abcd}$ is algebraically special.  
\end{lemma}

\begin{proof} 
Equation \eqref{eq:divmagic} is the case $\beta=1$ of \eqref{eq:DivPsialpha}. By Lemma \ref{lem:Eprop}, we have that $-1 < \PM{\cE} < 1$. Further, by Lemma \ref{lem:typeDchar} and Proposition \ref{prop:V(beta)} with $\beta=1$, $\PM{V} \geq 0$, with equality only if $\PM{\Wcal}_{abcd}$ is algebraically special. 
\end{proof}

\section{Integrating the divergence identity} \label{sec:int}

In this section we shall investigate the consequences of the divergence identity \eqref{eq:divmagic} presented in Lemma \ref{lem:magic}. By Lemma \ref{lem:Eprop}, $-1 < \PM{\cE} < 1$ and hence 
$\PM{\Psi}_a$ is smooth for $\lambda > 0$ and $\PM{\Fcal} > 0$, while it is singular on the set
\begin{align} \label{eq:cZdef}
\PM{\cZ} = \{\PM{\Fcal} = 0\} = \{\PM{\sbold} = 0\}
\end{align}
and on the fixed point set $\Fixed = \{\lambda = 0\}$ of the $S^1$-action. Integrating \eqref{eq:divmagic} over a bounded domain  $\Omega \subset \cM \setminus (\PM{\cZ} \cup \Fixed)$ 
yields 
\begin{align} \label{eq:magic-boundary} 
\int_\Omega \nabla^a \PM{\Psi}_a \, d\mu
= \int_{\partial \Omega} n^a \PM{\Psi}_a \, d\mu
\end{align} 
where $n^a$ is the outward pointing normal to $\partial \Omega$. We then consider a family $\Omega_\eps$ exhausting $\cM (\PM{\cZ} \cup \Fixed)$ and evaluate the limiting boundary contributions in terms of fixed point data. Using the $G$-signature formula, this is then used to give conditions under which  $(\cM, g_{ab})$ must be algebraically special, which in turn yields uniqueness results. 

\begin{remark} \label{rem:beta}
In this paper we shall consider applications of the identity \eqref{eq:DivPsialpha} for $\beta=1$. It is only in this case that the boundary terms at the fixed point set are local, in the sense that they do not depend on the pointwise values of globally defined quantities such as $\PM{\cE}$, see Lemma \ref{lem:rewrite}. 
\end{remark} 

\begin{lemma} \label{lem:Fcalprop} 
Near a nut $\nut$ with surface gravities $\kappa^1, \kappa^2$, we have
\begin{align}
\PM{\Fcal} 
= \PM{\Fcal}|_\nut + O(\lambda)
= 2|\kappa^1 \pm \kappa^2| + O(\lambda).
\end{align}
Near a bolt $\bolt$ with surface gravity $\kappa$ we have 
\begin{align}
\PM{\Fcal} 
= \PM{\Fcal}|_\bolt + O(\lambda)
= 2|\kappa| + O(\lambda).
\end{align}
\end{lemma} 

\begin{proof} 
From \eqref{eq:nablaFsq} we have $\nabla_{c}\PM{\Fcal}^2 = 0$ at points where $\xi=0$. So $\PM{\Fcal} = \PM{\Fcal}_\nut + O(\lambda)$ resp. $\PM{\Fcal} = \PM{\Fcal}|_\bolt + O(\lambda)$. Further, $\PM{\Fcal}|_\bolt$ is constant along a bolt. A computation using \eqref{eq:nablaxi(nut)} and \eqref{eq:nablaxi(bolt)} gives us $\PM{\Fcal}|_\nut = 2|\kappa^1 \pm \kappa^2|$ and $\PM{\Fcal}|_\bolt = 2|\kappa|$.
\end{proof}

\begin{prop} \label{prop:cZ-properties}
Let $\PM{\cZ}$ be the singular set given in \eqref{eq:cZdef} and let $\Fixed = \{ \lambda = 0 \}$ be the fixed point set of the $S^1$ action.  Assume that $(\cM, g_{ab})$ is not half-flat. Then
\begin{enumerate}
\item $\PM{\cZ}$ is compact,
\item $\PM{\cZ}$ and $\Fixed$ are disjoint.
\end{enumerate}
\end{prop}

\begin{proof}
Since $\PM{\cZ}$ is the zero set of a continuous function, it is closed. By \eqref{eq:sigmasq} and Lemma~\ref{lem:Eprop} we have
\begin{align} \label{Fcalasymp}
\PM{\Fcal}^2
= \lambda^{-1} \PM{\sigma}_{a} \PM{\sigma}^{a}
= \lambda^{-1}(\nabla_a\PM{\cE})(\nabla^a\PM{\cE})
=\PM{b}^2 r^{-4} + O^*(r^{-4-\delta})
\end{align}
and thus $r^4 \PM{\Fcal}^2 \to \PM{b}^2 \neq 0$ at infinity. Hence $\PM{\cZ}$ is bounded and therefore compact.

Suppose $\lambda = \PM{\Fcal} = 0$ at a point $p$. From \eqref{eq:nablaFsq} it follows that the derivative of $\PM{\Fcal}^2$ vanishes at $p$. Taking further derivatives of \eqref{eq:nablaFsq} and using \eqref{def:F_ab}, \eqref{def:Fcal_ab} and \eqref{eq:nablaFcal} we see that every term in every derivative of $\PM{\Fcal}^2$ contains either a factor $\xi_a$ or a factor $\PM\cF_{bd}$. It follows that all derivatives of $\PM{\Fcal}^2$ vanish at $p$. Since $(\cM, g_{ab})$ is Ricci-flat we may work in coordinates where all geometric quantities are analytic. It follows that $\PM{\Fcal}^2 = 0$ everywhere, which by Proposition~\ref{prop:SD-characterizations} contradicts the assumption that $(\cM, g_{ab})$ is not half-flat.
\end{proof}

\begin{prop} \label{prop:baer}
The set $\PM{\cZ}$ is countably 2-rectifiable, and therefore has Hausdorff dimension at most $2$.
\end{prop}
\begin{remark} 
The fact that $\PM{\cZ}$ is 2-rectifiable means that it can be written as a countable union of sets of the form $\Phi(\Xi)$, where $\Xi \subset \Reals^2$ is bounded and $\Phi: \Xi \to \cM$ is a Lipschitz map. In particular, the Hausdorff dimension of $\PM{\cZ}$ is at most 2.  
\end{remark} 

\begin{proof}
Taking the total anti-symmetrization respectively the trace of \eqref{eq:nablaFcal} and using the first Bianchi identity respectively the fact that $\PM\Wcal_{abcd}$ is trace-free, we find that the 2-form $\PM\cF_{ab}$ is closed and coclosed. From \cite[Main~Theorem]{Bar-nodal-sets} we thus know that $\PM\cZ$, which is the zero set of $\PM\cF_{ab}$, is countably 2-rectifiable. 
\end{proof}

Fix small $\delta, \eps > 0$. Since the Hausdorff dimension of $\PM{\cZ}$ is at most 2 there is a finite cover of $\PM{\cZ}$ by balls $B_{p_i}(r_i)$ of radius $r_i$ such that $\sum_i r_i^{2 + \delta} < \eps$. Define 
\begin{align} \label{Hausdorff-balls}
U_{\PM{\cZ},\eps} = \cup_i B_{p_i}(r_i). 
\end{align}
For $\eps > 0$ define 
\begin{align}
U_{\infty,\eps} = \{ r > \eps^{-1} \}, \quad
U_{\Fixed,\eps} = \{ \lambda < \eps \}.
\end{align} 
If $\eps$ is sufficiently small, the sets $U_{\PM{\cZ},\eps}$, $U_{\Fixed,\eps}$, $U_{\PM{\cZ},\eps}$ are all disjoint. Further, define 
\begin{align} 
\PM{\Omega}_\eps = 
\cM \setminus 
( U_{\infty,\eps} \cup U_{\Fixed,\eps} \cup U_{\PM{\cZ},\eps} ).
\end{align} 
The sets $\cM \setminus U_{\infty,\eps}$ give an exhaustion of $\cM$ as $\eps \to 0$ with the domains $U_{\infty,\eps}$ receding to infinity along the end. The domains $U_{\Fixed,\eps}$ surround the fixed point set $\Fixed$, and for $\eps \to 0$ converge to the nuts and bolts. Finally, the domains $U_{\PM{\cZ},\eps}$ surround the set $\PM{\cZ}$. It follows from the just mentioned facts that the family of domains $\PM{\Omega}_\eps$ forms an exhaustion of $\cM \setminus (\Fixed \cup \PM{\cZ})$. We have 
\begin{align} \label{boundary-decomposition}
\partial \PM\Omega_\eps = 
\partial U_{\infty,\eps} 
\cup \partial U_{\Fixed,\eps} 
\cup \partial U_{\PM\cZ,\eps}.
\end{align} 
By construction, $\PM\Psi_a$ is smooth in $\PM\Omega_\eps$. Therefore, in order to evaluate the right-hand side of \eqref{eq:magic-boundary}, we may consider fluxes of $\PM\Psi_a$ through each of these boundary components in \eqref{boundary-decomposition}, and take the limit $\eps \to 0$. The boundaries $\partial U_{\infty,\eps}$ and $\partial U_{\Fixed,\eps}$ are smooth hypersurfaces for small $\eps$, while the boundary $\partial U_{\PM\cZ,\eps}$ is smooth almost everywhere. The flux through $\partial U_{\infty,\eps}$ gives the contribution from infinity, while that through $\partial U_{\Fixed,\eps}$ gives the contribution from the nuts and bolts. Finally, the flux through $\partial U_{\PM\cZ,\eps}$ yields the contribution from the set where $\PM{\Fcal} = 0$. 

\begin{definition}[Length at infinity]\label{def:linfty}
For $p \in M$, let $\ell(p)$ denote the length of the $S^1$ orbit through $p$. The \emph{length at infinity} of $(\cM, g_{ab})$ is defined as  
\begin{align} \label{eq:linfty-def}
\ell_\infty = \liminf_{p \to \infty} \ell(p) .
\end{align} 
\end{definition} 

The remainder of this section is devoted to the proof of the following result.

\begin{prop}[Boundary terms] \label{prop:magbound} 
Let $(\cM, g_{ab})$ be an ALF $S^1$-instanton that is not half-flat. Let $O$ be the base of the boundary at infinity of $(\cM, g_{ab})$ and let $\chi[O]$ be its orbifold Euler characteristic, see Remark \ref{rem:seifert}.  Then
\begin{align}  \label{eq:magbound}
-2\pi \ell_\infty \chi[O]
+4\pi^2 \left(
\sum_{i=1}^{\nbolts} \frac{\chi[\bolt_i]}{|\kappa(\bolt_i)|}
\pm\sum_{i=1}^{\nnuts} \eps(\nut_i)
\frac{|\kappa^1(\nut_i) \pm \kappa^2(\nut_i)|}{|\kappa^1(\nut_i)\kappa^2(\nut_i)|}
\right)
\geq 0
\end{align}
with equality if and only if $\cM$ is algebraically special.
\end{prop} 

\begin{lemma} \label{lem:rewrite} 
Let $(\cM, g_{ab})$ be an ALF $S^1$-instanton that is not half-flat. Let $\PM{\Psi}_a$ be given by \eqref{eq:Psi-def}. Then there is a locally bounded one-form $\Xi_a$ such that  
\begin{align}
\PM{\Psi}_a 
= \pm \frac{\PM{\Fcal}}{2} J^T_a 
+ \frac{\nabla_a\PM{\Fcal}}{2\lambda} + \Xi_a .
\end{align} 
\end{lemma} 
\begin{proof}
We note 
\begin{align}
\frac{1 + \MP{\Ecal}}{1 - \PM{\Ecal}}
= 1 + \frac{\PM{\Ecal} + \MP{\Ecal}}{1 - \PM{\Ecal}} 
= 1 + \frac{2\lambda}{1 - \PM{\Ecal}}, 
\end{align}
and with $\PM{J}^a$ given by \eqref{eq:Jcurrdef} we have
\begin{align}
\PM{J}_{a} \cdot \frac{1}{(1 - \PM{\Ecal})^2}
&= 
\frac{\pm(1 \mp \P{\Ecal})^2 \M{\sigma}_{a} \mp (1 \pm \M{\Ecal})^2 \P{\sigma}_{a}}{2 \lambda^2} 
\cdot \frac{1}{(1 - \PM{\Ecal})^2}  \nonumber \\
&= 
\frac{1}{2 \lambda^2} \left(
\MP{\sigma}_{a} - \frac{(1 + \MP{\Ecal})^2}{(1 - \PM{\Ecal})^2} \PM{\sigma}_{a} 
\right)  \nonumber \\
&=
\frac{1}{2 \lambda^2} \left(
\mp 2\omega_a - \frac{4\lambda}{1 - \PM{\Ecal}} 
\left( 1 + \frac{\lambda}{1 - \PM{\Ecal}} \right)  \PM{\sigma}_{a}
\right).
\end{align}
Using \eqref{eq:MSscalar}, \eqref{eq:Psi-def}, and \eqref{eq:Jcurrdef} we then have
\begin{align}
\PM{\Psi}_a
&= 
\frac{1}{2} \PM{J}_a \PM{\sbold}
+ \frac{(1 + \MP{\Ecal}) (1 -  \PM{\Ecal})}{2\lambda} \nabla_a \PM{\sbold} \nonumber \\
&=
\frac{\PM{\Fcal}}{4 \lambda^2} \left(
\mp 2\omega_a 
- \frac{4\lambda}{1 - \PM{\Ecal}} \left( 1 + \frac{\lambda}{1 - \PM{\Ecal}} \right) \PM{\sigma}_{a}
\right) \nonumber \\
&\qquad 
+ \frac{(1+\MP{\Ecal}) (1-\PM{\Ecal})}{2\lambda} \cdot \frac{2\PM{\Fcal}}{(1-\PM{\Ecal})^3} \nabla_a \PM{\Ecal} \nonumber \\
&\qquad 
+ \frac{(1+\MP{\Ecal}) (1-\PM{\Ecal})}{2\lambda} \cdot \frac{\nabla_a \PM{\Fcal}}{(1 -  \PM{\Ecal})^2} \nonumber \\
&= 
\mp\frac{\PM{\Fcal}}{2} \cdot \frac{\omega_a}{\lambda^2} \nonumber \\
&\qquad
+\frac{\PM{\Fcal}}{\lambda(1-\PM{\Ecal})} 
\left(
- \left( 1 + \frac{\lambda}{1 - \PM{\Ecal}} \right) 
+ \frac{1 + \MP{\Ecal}}{1-\PM{\Ecal}} 
\right) \PM{\sigma}_{a} \nonumber \\
&\qquad 
+ \frac{1 + \MP{\Ecal}}{1 -  \PM{\Ecal}} \cdot \frac{\nabla_a\PM{\Fcal}}{2\lambda} \nonumber \\
&= 
\pm\frac{\PM{\Fcal}}{2} J^T_a 
+ \frac{\PM{\Fcal}}{(1-\PM{\Ecal})^2} \PM{\sigma}_{a} 
+ \left( 1 + \frac{2\lambda}{1 - \PM{\Ecal}}\right) \frac{\nabla_a\PM{\Fcal}}{2\lambda}. 
\end{align}
The result follows. 
\end{proof} 

\subsection{Contribution from the fixed point set}
\begin{lemma}[Boundary term at the nuts]\label{lem:bdrnut}
Let $\nut$ be a nut with surface gravities $\kappa^1, \kappa^2$. Let $U_{\nut,\epsilon}$ be the connected neighbourhood of $\nut$ where $\lambda < \epsilon$ and let $S_{\nut,\epsilon} = \partial U_{\nut,\epsilon}$ be its boundary with inward pointing unit normal $n^a$. Then
\begin{align}
\lim_{\epsilon \to 0} 
\int_{S_{\nut,\epsilon}} n^a \PM{\Psi}_a \, d\mu
= \pm 4\pi^2 \frac{|\kappa^1 \pm \kappa^2|}{\kappa^1\kappa^2}. 
\end{align} 
\end{lemma} 

\begin{proof} 
By Lemma~\ref{lem:Fcalprop}, Proposition~\ref{prop:charges}, and Lemma \ref{lem:rewrite} we have
\begin{align}
\int_{S_{\nut,\epsilon}} n^a \PM{\Psi}_a \, d\mu
&=
\int_{S_{\nut,\epsilon}} n^a  
\left( \pm\frac{\PM{\Fcal}}{2} J^T_a 
+ \frac{\nabla_a\PM{\Fcal}}{2\lambda}
+ \Xi_a \right) \, d\mu \nonumber \\
&=
\int_{S_{\nut,\epsilon}} n^a  
\left(
\pm \frac{2|\kappa^1 \pm \kappa^2| + O(\lambda)}{2} J^T_a 
+ O(\lambda^{-1})
\right) \, d\mu \nonumber \\
&=
\pm |\kappa^1 \pm \kappa^2|  \cdot 8\pi N(\nut) + O(\epsilon) 
+ \int_{S_{\nut,\epsilon}} O(\lambda^{-1}) \, d\mu \nonumber \\
&=
\pm 4\pi^2 \frac{|\kappa^1 \pm \kappa^2|}{\kappa^1\kappa^2}  
+ O(\epsilon^{1/2}), 
\end{align} 
and the result follows.
\end{proof} 

\begin{lemma}[Boundary term at the bolts] \label{lem:bdrybolt}
Let $\bolt$ be a bolt with surface gravity $\kappa$. Let $U_{\bolt,\epsilon}$ be the connected neighbourhood of $\bolt$ where $\lambda < \epsilon$ and let $S_{\bolt,\epsilon} = \partial U_{\bolt,\epsilon}$ be its boundary with inward pointing unit normal $n^a$. Then 
\begin{align}
\lim_{\eps \to 0} 
\int_{S_{\bolt,\epsilon}} n^a \PM{\Psi}_a \, d\mu
= \frac{4\pi^2}{|\kappa|} \chi[\bolt] .
\end{align}
\end{lemma} 

\begin{proof} 
By Lemma~\ref{lem:Fcalprop}, Proposition~\ref{prop:charges}, and Lemma \ref{lem:rewrite} we have 
\begin{align}
\int_{S_{\bolt,\epsilon}} n^a \PM{\Psi}_a \, d\mu
&=
\int_{S_{\bolt,\epsilon}} n^a  
\left(
\pm\frac{\PM{\Fcal}}{2} J^T_a
+ \frac{\nabla_a \PM{\Fcal}}{2\lambda}
 + \Xi_a \right) \, d\mu \nonumber \\
&=
\pm\frac{2|\kappa| + O(\epsilon)}{2} 8\pi N(\bolt)
+ O(\epsilon^{1/2})
\nonumber \\
&\qquad
+\frac{1 + O(\epsilon)}{2\epsilon} \int_{S_{\bolt,\epsilon}}
\frac{1}{2\PM{\Fcal}} n^a \nabla_a \PM{\Fcal}^2 \, d\mu
\nonumber \\
&=
\mp \left( |\kappa| + O(\epsilon) \right)
\cdot \frac{8\pi}{8\kappa^2} \int_\bolt  \cR^{\parallel\perp} d\mu
+ O(\epsilon^{1/2})
\nonumber \\
&\qquad
+\frac{1 + O(\epsilon)}{4\epsilon (2|\kappa| + O(\epsilon) )} 
\int_{S_{\bolt,\epsilon}} n^a \nabla_a \PM{\Fcal}^2 \, d\mu.
\end{align}
We use \eqref{eq:Fcalepsperpepspara} and \eqref{eq:epsperpepspara} to compute on $\bolt$,
\begin{align}
\PM{\Wcal}_{abcd} \PM{\Fcal}^{cd}
&=
(W_{abcd} \pm \frac{1}{2} \eps_{cd}{}^{ef} W_{abef})
\kappa (\eps^\perp{}^{cd} \pm \eps^\parallel{}^{cd}) \nonumber \\
&=
\kappa(R_{abcd} \pm \frac{1}{2} R_{abef} \eps_{cd}{}^{ef} )
(\eps^\perp{}^{cd} \pm \eps^\parallel{}^{cd}) \nonumber \\
&=
\kappa(R_{abcd}\eps^\perp{}^{cd} \pm \frac{1}{2}R_{abef}\eps_{cd}{}^{ef}\eps^\perp{}^{cd}) \nonumber \\
&\qquad
\pm
\kappa(R_{abcd}\eps^\parallel{}^{cd} \pm \frac{1}{2}R_{abef} \eps_{cd}{}^{ef}\eps^\parallel{}^{cd} ) \nonumber \\
&=
\kappa(R_{abcd}\eps^\perp{}^{cd} \pm R_{abef}\eps^\parallel{}^{ef}) \nonumber \\
&\qquad
\pm
\kappa(R_{abcd}\eps^\parallel{}^{cd} \pm R_{abef} \eps^\perp{}^{ef} ) \nonumber \\
&=
2\kappa 
( R_{abcd}\eps^\perp{}^{cd} \pm R_{abcd}\eps^\parallel{}^{cd} ) ,
\end{align}
and with \eqref{eq:Rparaparagauss} and \eqref{eq:Rperpperpgauss} we get
\begin{align}
\PM{\Wcal}_{abcd} \PM{\Fcal}^{ab} \PM{\Fcal}^{cd}
&=
2\kappa^2 
( R_{abcd}\eps^\perp{}^{cd} \pm R_{abcd}\eps^\parallel{}^{cd} )
(\eps^\perp{}^{ab} \pm \eps^\parallel{}^{ab}) \nonumber \\
&= 
2\kappa^2 
( R_{abcd}\eps^\perp{}^{ab} \eps^\perp{}^{cd} 
\pm R_{abcd}\eps^\perp{}^{ab} \eps^\parallel{}^{cd} ) \nonumber \\
&\qquad 
\pm 2\kappa^2 
( R_{abcd}\eps^\parallel{}^{ab} \eps^\perp{}^{cd} 
\pm R_{abcd}\eps^\parallel{}^{ab} \eps^\parallel{}^{cd} ) \nonumber \\
&=
2\kappa^2 
( 2\cR^{\perp\perp} \pm 2\cR^{\perp\parallel} ) 
\pm 2\kappa^2 
( 2\cR^{\perp\parallel} \pm 2\cR^{\parallel\parallel} ) \nonumber \\
&=
4\kappa^2 (\cR^{\perp\perp} + \cR^{\parallel\parallel} )
\pm 8\kappa^2 \cR^{\perp\parallel} \nonumber \\
&=
8\kappa^2 ( R^\bolt \pm \cR^{\perp\parallel} ). 
\end{align}
Taking into account the fact that $n^a$ is the inward pointing unit normal to $S_{\bolt,\eps}$ with respect to $U_{\bolt,\eps}$, we have using 
\eqref{eq:laplaceFsq}, 
\begin{align}
\int_{S_{\bolt,\epsilon}} n^a \nabla_a \PM{\Fcal}^2 \, d\mu
&=
-\int_{U_{\bolt,\epsilon}} \nabla^a \nabla_a \PM{\Fcal}^2 \, d\mu \nonumber \\
&=
-\int_{U_{\bolt,\epsilon}}
\left(
- \PM{\Wcal}_{abcd} \PM{\Fcal}^{ab} \PM{\Fcal}^{cd}
+ \frac{1}{2} \lambda  \PM{\Wcal}^2 
\right)
\, d\mu \nonumber \\
&=
\int_{U_{\bolt,\epsilon}}
\PM{\Wcal}_{abcd} \PM{\Fcal}^{ab} \PM{\Fcal}^{cd} \, d\mu
- \frac{\epsilon}{2}
\int_{U_{\bolt,\epsilon}}  \PM{\Wcal}^2  \, d\mu \nonumber \\
&=
\frac{\pi \eps}{ \kappa^2}
\int_\bolt 
8\kappa^2 ( R^\bolt \pm \cR^{\perp\parallel} ) \, d\mu
+ \frac{\epsilon}{2} O(\epsilon^{1/2}).
\end{align}
In the last step here, the integral over the tube $U_{\bolt,\epsilon}$ is approximated with the integral over $\bolt$ multiplied with the area $\frac{\pi \eps}{ \kappa^2}$ of the normal disk. Together we find 
\begin{align}
\int_{S_{\bolt,\epsilon}} n^a \PM\Psi_a \, d\mu
&=
\mp \pi\frac{|\kappa| + O(\epsilon)}{\kappa^2} 
\int_\bolt \cR^{\parallel\perp} \, d\mu
+ O(\epsilon^{1/2})
\nonumber \\
&\qquad
+\frac{1 + O(\epsilon)}{4\epsilon (2|\kappa| + O(\epsilon) )} 
\left(
\frac{\pi \eps}{ \kappa^2}
\int_\bolt
8\kappa^2 ( R^\bolt \pm \cR^{\perp\parallel} ) \, d\mu
+ \frac{\epsilon}{2} O(\epsilon^{1/2})
\right) \nonumber \\
&=
\mp \pi\frac{|\kappa| + O(\epsilon)}{\kappa^2} \int_\bolt \cR^{\parallel\perp} \, d\mu
+ O(\epsilon^{1/2})
\nonumber \\
&\qquad
+ \pi\frac{1 + O(\epsilon)}{|\kappa| + O(\epsilon)}
\left(
\int_\bolt
( R^\bolt \pm \cR^{\perp\parallel} ) \, d\mu
+ \frac{\epsilon}{2} O(\epsilon^{1/2})
\right).
\end{align}
Finally, the Gauss--Bonnet theorem tells us that
\begin{align}
\lim_{\eps \to 0} 
\int_{S_{\bolt,\epsilon}} n^a \PM\Psi_a \, d\mu
&=
\mp \frac{\pi}{|\kappa|} \int_\bolt \cR^{\parallel\perp} \, d\mu
+ \frac{\pi}{|\kappa|} 
\int_\bolt (R^\bolt \pm \cR^{\parallel\perp}) \, d\mu \nonumber \\
&= \frac{\pi}{|\kappa|} \int_\bolt R^\bolt \, d\mu \nonumber \\
&= \frac{2\pi}{|\kappa|} \int_\bolt \mathcal{K}^\bolt \, d\mu \nonumber \\
&= \frac{4\pi^2}{|\kappa|} \chi[\bolt]. 
\end{align}
\end{proof}

\subsection{Contribution from infinity} \label{sec:infterm}

\begin{lemma}[Length at infinity] \label{lem:ellinfty} 
Let $(M,g)$ be an ALF-$A_k$ $S^1$-instanton, with $S^1$ action generated by $\xi$, with $\lambda \to 1$ at $\infty$. 
Let  $\Pi$ be the minimum period of the linearized $S^1$ action at the fixed point set, 
\begin{align} \label{eq:Pi-ineq}
\Pi = \min \left \{ \frac{2\pi}{|\kappa(\nut)|}, \frac{2\pi}{|\kappa(\bolt)|} \right \},  
\end{align} 
where the minimum is taken over all nuts and bolts in the fixed point set $F$, and all surface gravities. 
The length at infinity $\ell_\infty$, see Definition \ref{def:linfty}, satisfies 
\begin{align} \label{eq:ellinfty}
\ell_\infty \geq \Pi. 
\end{align} 
\end{lemma} 

\begin{proof} 
Assume for a contradiction there is an $\eps > 0$ such that for any $r_0$ sufficiently large there is $p \in M$, $r(p) > r_0$ with with $\ell(p) < (1-\eps)\Pi$. We may without loss of generality assume $\lambda(p) > 1-\eps$ for $r(p) > r_0$, by choosing $r_0$ large enough. 

Since $\lambda \to 1$ at $\infty$, $\Fixed$ is compact. Hence, there is a minimizing geodesic $\gamma$ connecting $p$ to  $\Fixed$. Let $\gamma(0) \in \Fixed$ be the starting point of $\gamma$. The period of the $S^1$ action on $\gamma$ is determined by the period of the linearized action at $x$ on the initial velocity vector $\dot \gamma(0)$. Let the period of the $S^1$ action on $\gamma$ be $\Pi_\gamma$. In view of \eqref{eq:Pi-ineq}, $\Pi_\gamma \geq \Pi$. We have that 
\begin{align} 
(1-\eps) \Pi > \ell(p) =\lambda^{1/2}(p) \Pi_\gamma > (1-\eps/2)\Pi
\end{align} 
which is a contradiction. This proves the lemma.
\end{proof} 

\begin{lemma}[The boundary term at infinity] \label{lem:bdrinfty} 
Let $S_{\infty,\epsilon} = \partial U_{\infty,\epsilon} = \{ r = \eps^{-1}\}$ with outward outward-pointing unit normal $n^a$. Then
\begin{align}
\lim_{\eps \to 0}
\int_{S_{\infty,\epsilon}} n^a\PM{\Psi}_a \, d\mu
\leq -2\pi \ell_\infty \chi[O]. 
\end{align}
\end{lemma} 

\begin{proof} 
By \eqref{eq:cE-expand} we have
\begin{align}
\PM{\cE} &= 1 + \PM{b} r^{-1} + O^*(r^{-1- \delta}), \nonumber \\
\PM{\sigma}_{a} = \nabla_a \PM{\cE}
&= -\PM{b} r^{-2} \nabla_a r + O^*(r^{-2- \delta}),
\end{align}
as $r \to \infty$. By \eqref{Fcalasymp} we further have 
\begin{align}
\PM{\Fcal}
&=|\PM{b}| r^{-2} + O^*(r^{-2-\delta})
=-\PM{b} r^{-2} + O^*(r^{-2-\delta}), \nonumber \\
\nabla_a \PM{\Fcal}
&= 2\PM{b} r^{-3} \nabla_a r + O^*(r^{-3-\delta}).
\end{align}
We find that
\begin{align}
\PM{\Psi}_a
&=
\frac{1}{2} \PM{J}_a \frac{\PM{\Fcal}}{(1-\PM{\Ecal})^2}
+ \frac{(1 + \MP{\Ecal}) (1 -  \PM{\Ecal})}{2\lambda}
\nabla_a \left(\frac{\PM{\Fcal}}{(1-\PM{\Ecal})^2} \right)\nonumber \\
&=
\frac{1}{4 \lambda^2} \left(
\MP{\sigma}_{a} - \frac{(1 + \MP{\Ecal})^2}{(1 - \PM{\Ecal})^2} \PM{\sigma}_{a}
\right) \PM{\Fcal} \nonumber \\
&\qquad
+ \frac{1 + \MP{\Ecal}}{2\lambda(1-\PM{\Ecal})}
\nabla_a \PM{\Fcal}
+ \frac{1 + \MP{\Ecal}}{\lambda(1-\PM{\Ecal})^2}
\PM{\Fcal} \nabla_a \PM{\cE} \nonumber \\
&=
- r^{-2} \nabla_a r + 2r^{-2} \nabla_a r - 2r^{-2} \nabla_a r + O^*(r^{-2-\delta})\nonumber \\
&=
- r^{-2} \nabla_a r + O^*(r^{-2-\delta}) \label{eq:Psiexpansion}.
\end{align}
We have 
$n^a=\frac{\nabla^a r}{|\nabla r|}$, where $|\nabla r| \to 1$ as follows from \eqref{eq:gring} and \eqref{eq:gringh}. From \eqref{eq:Psiexpansion} we find that 
\begin{align}
\int_{S_{\infty,\epsilon}} n^a\PM{\Psi}_a \, d\mu
&= \int_{S_{\infty,\epsilon}}(-\eps^2 + O(\eps^{2 + \delta})) \, d\mu \nonumber \\
&= - (1 + O(\eps^{\delta})) \eps^2 \mathrm{Vol}(S_{\infty,\epsilon},g) \nonumber \\
&= - (1 + O(\eps^{\delta})) \eps^2 \mathrm{Vol}(L, \sigma_{\eps^{-1}}) %
\end{align}

The manifold $L$ is a circle fibration over the orbifold $O$. From Definition~\ref{def:BG} we have that the transverse metric $\gamma$ gives a metric on $O$ with constant curvature $+1$. With respect to this fibration, the metric $\sigma_r = r^2 \gamma + \eta^2$, and its volume form is $d\mu^{\sigma_r} = r^2 d\mu^{\gamma} \wedge\eta$. By Fubini's theorem and the orbifold Gauss--Bonnet theorem \cite[\ S 2]{MR705527}, the volume of $L$ is 
\begin{align}
\mathrm{Vol}(L, \sigma_{\eps^{-1}}) 
&= \int_O \left( \eps^{-2} \int_{\text{fiber}} \eta \right) d\mu^{\gamma} \nonumber \\
&\geq 
\eps^{-2} \ell_\infty \int_O 1 \, d\mu^{\gamma} \nonumber \\
&=
\eps^{-2} \ell_\infty \int_O \mathcal{K}^O \, d\mu^{\gamma} \nonumber \\
&= 2\pi \eps^{-2} \ell_\infty \chi[O].
\end{align}
We conclude that
\begin{align}
\lim_{\eps \to 0}
\int_{S_{\infty,\epsilon}} n^a\PM{\Psi}_a \, d\mu
\leq -2\pi \ell_\infty \chi[O].
\end{align}
\end{proof}

\subsection{The singular term} \label{sec:singular} 

\begin{lemma}[The singular term]\label{lem:bdrsing} 
\begin{align} 
\lim_{\eps \to 0}
\int_{\partial U_{\PM\cZ,\eps}} n^a \PM{\Psi}_a \, d\mu  
= 0.
\end{align} 
\end{lemma} 

\begin{proof}
Proposition~\ref{prop:cZ-properties} tells us that $\lambda$ is bounded away from zero near $\PM\cZ$. From \eqref{eq:Psi-def} and \eqref{eq:MSscalar} we therefore see that $|n^a \PM{\Psi}_a| \leq C$ independently of $\eps > 0$ near $\PM\cZ$. Thus 
\begin{align} 
\left| \int_{\partial U_{\PM\cZ,\eps}} n^a \PM{\Psi}_a \, d\mu \right| 
\leq
C \operatorname{Vol} \left( \partial U_{\PM\cZ,\eps} \right).
\end{align}
By \eqref{Hausdorff-balls} we have for sufficiently small $\delta > 0$ 
\begin{align}
\operatorname{Vol} \left( \partial U_{\PM\cZ,\eps} \right)
\leq D \sum_i r_i^3 
\leq D \sum_i r_i^{2 + \delta} 
< D \eps,
\end{align}
and the claim follows.
\end{proof} 

\begin{remark}
The singular set $\PM{\cZ}$ was also an issue in the black hole uniqueness proofs carried out on the orbit space in the hypersurface-orthogonal case, as sketched in Section \ref{sec:intro}. While in his proof Israel \cite{1967PhRv..164.1776I} assumed the absence of such a set, the latter was admitted and subject of a technically subtle discussion in M\"uller zum Hagen et. al. \cite{MR398432}.
On the other hand, we are not aware of suitable generalizations of Robinson's identities \cite{1977GReGr...8..695R}, which avoid the singular set, to the non-hypersurface orthogonal case.  
\end{remark}

\subsection{Proof of Proposition \ref{prop:magbound}}
Proposition \ref{prop:magbound} follows from Lemmas \ref{lem:magic}, \ref{lem:bdrnut}, \ref{lem:bdrybolt}, \ref{lem:bdrinfty}, \ref{lem:bdrsing}.

\section{Proof of the main theorem} \label{sec:bound}

In this section, we will analyze the terms in the divergence identity derived in Proposition~\ref{prop:magbound} and prove our main theorem. The first step is the following corollary to Proposition~\ref{prop:magbound}.

\begin{definition}
For a nut $\nut$ with fixed point data $\{\eps, a,b\}$ with $a \leq b$, let
\begin{align} 
\PM{Z}(\nut) = \pm \eps \frac{1}{a} + \frac{1}{b}. 
\end{align} 
\end{definition}

\begin{cor} 
\label{cor:magdim}
Let $(\cM, g_{ab})$ be an ALF-$A_k$ $S^1$-instanton which is not half-flat. Let $w \geq 1$ be the maximal weight of all nuts in $\cM$. If $\cM$ has no nuts, we set $w= 1$. It holds that 
\begin{align} \label{eq:magdim}
-\frac{2}{w} + \chi[\cM] - \nnuts + \sum_{i=1}^{\nnuts} \PM{Z}(\nut_i) 
\geq 0
\end{align}
with equality if and only if $\PM{\Wcal}_{abcd}$ is algebraically special.  
\end{cor} 

\begin{proof} 
Since $(\cM, g_{ab})$ is ALF-$A_k$ we have $\chi[O] = 2$. Let $G$ be as in Lemma \ref{lem:generic}. Divide \eqref{eq:magbound} by $4\pi^2 G$. The corollary then follows from \eqref{eq:chinutbolt} and Lemma~\ref{lem:ellinfty}.
\end{proof} 

\begin{lemma}\label{lem:Zpm-est} 
Let $\nut$ be a nut. If $\epsilon(\nut)=1$, then $\P{Z}(\nut)\leq2$ and $\M{Z}(\nut)\leq 0$. If $\epsilon(\nut)=-1$, then $\M{Z}(\nut)\leq 2$ and $\P{Z}(\nut)\leq 0$. For any of the inequalities, equality holds if and only if the weights of $\nut$ are $\{1,1\}$.
\end{lemma}

\begin{proof}
Let $\nut$ have weights $a,b$, where $a\leq b$. If $\epsilon(\nut)=1$, then 
\begin{align}
\P{Z}(\nut)=\frac{1}{a}+\frac{1}{b},
\end{align} 
and since $a$ and $b$ are positive integers, we have $\P{Z}(\nut)\leq 2$ with equality if and only if $a=b=1$. Since $a\leq b$ we see that 
\begin{align}
\M{Z}(\nut)=-\frac{1}{a}+\frac{1}{b}\leq 0
\end{align} 
where equality occurs if and only if $a=b$, which happens only when  $a=b=1$ since $a$ and $b$ are coprime. The reasoning in the case $\epsilon(\nut)=-1$ is similar.
\end{proof}

We are now ready to apply inequality~\eqref{eq:magdim} to known spaces with ALF-$A_k$ $S^1$-instanton metrics. We first consider the case with Kerr topology. 

\begin{lemma} \label{lem:Kerr-eq}
Let $(\cM,g_{ab})$ be an AF $S^1$-instanton on $S^4 \setminus S^1$. Then equality holds in \eqref{eq:magdim}, for both signs. 
\end{lemma}

\begin{proof}
In the present case we have $e=0$, $\chi[\cM]=2$ and $\sign[\cM]=0$. Equation \eqref{eq:G-sign-g0} with $e=0$ tells us that the number of positively oriented nuts is equal to the number of negatively oriented nuts. We denote this number by $k$. Corollary \ref{cor:magdim} yields 
\begin{align} \label{eq:lowbdZ}
-\frac{2}{w} + 2 - 2k + \sum_{i=1}^{\nnuts}\PM{Z}(\nut_i) \geq 0.
\end{align} 
In case there are no nuts, $w=1$ and $k=0$, and we have equality in \eqref{eq:lowbdZ} and hence in 
\eqref{eq:magdim}. 

We now assume $k > 0$. 
In the case $w=1$, all weights are $1$, so we get $\PM{Z}(\nut_i)=\pm\epsilon(\nut_i)+1$ for all $i$, and thus the left hand side of \eqref{eq:lowbdZ} is equal to 
\begin{align} 
-2 +2-2k + \sum_{i=1}^{\nnuts}(\pm\epsilon(\nut_i)+1)
=
\pm\sum_{i=1}^{\nnuts}\epsilon(\nut_i) 
= 0 .
\end{align} 
This means that equality holds in \eqref{eq:lowbdZ} and \eqref{eq:magdim},  for both signs. 

Now consider the case $w>1$, and assume without loss of generality that $\nut_1$ has nut data $\{\epsilon(\nut_1),a,w\}$ for some positive integer $a$. By Lemma~\ref{lem:3.6}, there exists another nut, with nut data either $\{-\epsilon(\nut_1),a,w\}$ or $\{\epsilon(\nut_1),b,w\}$, where $a+b=w$; without loss of generality, assume that this nut is $\nut_2$.
In the case where $\nut_2$ has nut data $\{-\epsilon(\nut_1),a,w\}$, we have 
\begin{align}
\PM{Z}(\nut_1) + \PM{Z}(\nut_2)
=
\pm\frac{\epsilon(\nut_1)}{a} + \frac{1}{w}\mp\frac{\epsilon(\nut_1)}{a} + \frac{1}{w}
=
\frac{2}{w},
\end{align} 
and then \eqref{eq:lowbdZ} is equivalent to 
\begin{align}
2 - 2k + \sum_{i=3}^{\nnuts}\PM{Z}(\nut_i) \geq 0.
\end{align} 
On the other hand, since the set $\{\nut_3,\dots,\nut_{\nnuts}\}$ contains $k-1$ nuts of each orientation, Lemma~\ref{lem:Zpm-est} implies that 
\begin{align} \label{eq:upbdZ}
2 - 2k + \sum_{i=3}^{\nnuts}\PM{Z}(\nut_i) \leq 0
\end{align} 
which means that equality must hold in view of Lemma~\ref{lem:Zpm-est}. Consider now the case where $\nut_2$ has nut data $\{\epsilon(\nut_1),b,w\}$, with $a+b=w$. Then 
\begin{align}
\PM{Z}(\nut_1) + \PM{Z}(\nut_2)=\frac{2}{w}\pm\epsilon(\nut_1)\left (\frac{1}{a}+\frac{1}{b} \right),
\end{align} 
so that \eqref{eq:lowbdZ} is equivalent to 
\begin{align} \label{eq:newbd}
\pm\epsilon(\nut_1)\left(\frac{1}{a}+\frac{1}{b}\right)
+2 - 2k + \sum_{i=3}^{\nnuts} \PM{Z}(\nut_i) \geq 0.
\end{align} 
From \eqref{eq:G-sign-g0} we find that the set $\{\nut_3,\dots,\nut_{\nnuts}\}$ must contain $k-2$ ($k \ge 2$) nuts with the same orientation as $\nut_1$, and $k$ nuts with the opposite orientation. Assuming that $\epsilon(\nut_3)=\epsilon(\nut_4)=-\epsilon(\nut_1)$, Lemma~\ref{lem:Zpm-est} tells us that
\begin{align}
\sum_{i=5}^{\nnuts} \PM{Z}(\nut_i) \leq 2k-4.
\end{align}
In case $\epsilon(\nut_1)=1$, thus, the $+$ version of \eqref{eq:newbd} implies that 
\begin{align} 
0\leq\frac{1}{a}+\frac{1}{b} + \P{Z}(\nut_3) + \P{Z}(\nut_4)-2\leq\frac{1}{a}+\frac{1}{b}-2\leq 0,
\end{align}
where the middle inequality uses again Lemma \ref{lem:Zpm-est}. Thus, $a=b=1$ and $\P{Z}(\nut_3) = \P{Z}(\nut_4)=0$, which, once again by Lemma \ref{lem:Zpm-est}, means that $\nut_3$ and $\nut_4$ have weights $\{1,1\}$. In the case $\epsilon(\nut_1)=-1$, similarly, the $-$ version of \eqref{eq:newbd} implies that
\begin{align} 
0 \leq 
\frac{1}{a}+\frac{1}{b}+\M{Z}(\nut_3) + \M{Z}(\nut_4)-2
\leq
\frac{1}{a}+\frac{1}{b}-2\leq 0. 
\end{align} 
In either case, thus, $a=b=1$, and $\nut_3$ and $\nut_4$ have weights $\{1,1\}$. Hence the left hand side of \eqref{eq:lowbdZ} is equal to  
\begin{align}
&\pm 2\epsilon(\nut_1) + 2(\mp\epsilon(\nut_1)+1) 
+2 -2k +\sum_{i=5}^{\nnuts}\PM{Z}(\nut_i)  \nonumber \\ 
&\qquad 
\leq 2 + 2 - 2k + (2k-4) = 0, 
\end{align}
which finishes the proof.
\end{proof}

We next consider the case of Taub-bolt topology, which turns out to be rather similar to the Kerr case. 

\begin{lemma} \label{lem:Tb-eq}
Let $(\cM,g_{ab})$ be an ALF-$A_0$ $S^1$-instanton on $\mathbb{C}P^2\setminus\{\text{\rm pt.}\}$. Then equality holds in \eqref{eq:magdim}, for both signs.
\end{lemma}

\begin{proof}
From the assumption on the topology, we have $\chi[\cM]=2$ and $\sign[\cM]=1$, and since $\cM$ is ALF-$A_0$, the circle fibration at infinity has Euler number $e=1$. By \eqref{eq:G-sign-g0} we see that, as in the proof of the previous lemma, there is an equal number of positively and negatively oriented nuts. The rest of that proof goes through without any modification.
\end{proof}

Finally, we consider the case of an $S^1$-instanton with the same topology as the Chen--Teo instanton. 

\begin{lemma} \label{lem:CT-eq}
Let $(\cM,g_{ab})$ be an AF $S^1$-instanton on $\mathbb{C}P^2\setminus S^1$. Then equality holds in \eqref{eq:magdim}, for $-$.  
\end{lemma}

\begin{proof}
Let $\PM{\Phi}$ be the left hand side of \eqref{eq:magdim}, so $\PM{\Phi}\geq 0$ by Corollary~\ref{cor:magdim}. We wish to show that $\M{\Phi}=0$. By the assumption on the topology, we have $\chi[\cM]=3$ and $\sign[\cM]=1$. From \eqref{eq:G-sign-g0} we see that 
\begin{align}
\sum_{i=1}^{\nnuts}\epsilon(\nut_i) = 1,
\end{align} 
and thus the numbers of positively oriented and negatively oriented nuts are $k+1$ and $k$, respectively, for some integer $k\geq 0$. If all nuts  have weights $\{1,1\}$, then $\M{\Phi}=0$ and we are done, so assume that $w>1$, where $w$ is the largest weight, and assume without loss of generality that $\nut_1$ has this weight; then $\nut_1$ has nut data $\{\epsilon(\nut_1),b,w\}$ for some $b$. If there exists another nut with nut data $\{-\epsilon(\nut_1),b,w\}$, then the same argument as in the proof of \eqref{eq:upbdZ} works to show that $\M{\Phi}\leq0$, since there are $k-1$ remaining negatively oriented nuts in this case.

Thus, assume that there exists no nut with nut data $\{-\epsilon(\nut_1),b,w\}$. By Lemma~\ref{lem:3.5} there exists another nut with nut data $\{\epsilon(\nut_1),a,w\}$, where $a+b=w$. Assume without loss of generality that this nut is $P_2$, and that $b\geq a$. If $a=b=1$ and $\epsilon(\nut_1)=1$, then using Lemma \ref{lem:Zpm-est} we see that 
\begin{align}
\M{\Phi}
= 
-2 +2 -2k +\sum_{i=3}^{\nnuts} \M{Z}(\nut_i)
\leq -2 +2 - 2k +2k 
= 0.
\end{align} 
If $a=b=1$ and $\epsilon(\nut_1)=-1$, we similarly find 
\begin{align} 
\M{\Phi}
=
2 + 2 -2k + \sum_{i=3}^{\nnuts} \M{Z}(\nut_i)
\leq 
2 +2 -2k + 2(k-2)
= 0.
\end{align} 

Henceforth, we can thus assume that $b>a$. By applying Lemma~\ref{lem:3.5} to the nut $\nut_1$, there is a nut with data either $\{\epsilon(\nut_1),b,c\}$, where $a+b\equiv -c\mod{b}$, or data $\{-\epsilon(\nut_1),b,c\}$, where $a+b\equiv c\mod{b}$. Since $b$ is equal to neither $a$ nor $a+b$, this nut cannot be $\nut_2$, so we can assume without loss of generality that it is $\nut_3$. If $\epsilon(\nut_3)=\epsilon(\nut_1)$, then since $a+c\equiv a+b+c\equiv 0\mod{b}$ and since $0<a+c<b+(a+b)<3b$, we have $c\in\{b-a,2b-a\}$. On the other hand, if $\epsilon(\nut_3)=-\epsilon(\nut_1)$, then since $a-c\equiv a+b-c\equiv 0\mod{b}$ and since $a + b$ is by assumption the highest weight in the configuration, we have $c\in\{a+b,a\}$. However, the case $c=a+b$ in which $\nut_3$ would have nut data $\{-\epsilon(\nut_1),b,w\}$  has also been excluded by assumption above. This implies $c = a$.

Together with the possibilities of the choice of $\epsilon(\nut_1)$, we are left with six cases each of which will be considered separately. In the list below, the ordering of the nuts is $P_1$, $P_2$ and $P_3$, and the weights of each nut are displayed in increasing order.

\subsection*{Case: $\{+,b,a+b\},\{+,a,a+b\} \{+,b-a ,b\}$}
By Lemma~\ref{lem:Zpm-est} we have
\begin{align} 
\P{\Phi}
&=   
\frac{1}{a}+\frac{2}{b}+\frac{1}{b-a} +2-2k +\sum_{i=4}^{\nnuts}\P{Z}(\nut_i)
\leq 
\frac{1}{a}+\frac{2}{b}+\frac{1}{b-a}-2, \label{case1+}\\
\M{\Phi}
&= 
-\frac{1}{a} -\frac{1}{b-a} +2-2k +\sum_{i=4}^{\nnuts}\M{Z}(\nut_i)
\leq 
-\frac{1}{a} -\frac{1}{b-a} + 2. \label{case1-} 
\end{align}
The right hand side of \eqref{case1+} is negative for $b>3$ which is thus impossible. The remaining cases $(a,b) =(1,2)$, $(a,b)=(1,3)$ and $(a,b)=(2,3)$ will be handled separately. When $(a,b) = (1,2)$, \eqref{case1-} yields $\M{\Phi} = 0$ and we are done.
 
When $(a,b)=(1,3)$, we have the nut configuration 
\begin{align} 
\{+,3,4\}, \quad \{+,1,4\}, \quad \{+,2,3\}, 
\end{align} 
and since the weight $2$ occurs only once in this configuration, by Lemma~\ref{lem:2.6} there must exist another nut, say $\nut_4$, with nut data $\{\epsilon(\nut_4),2,c\}$ for some $c$. If $\epsilon(\nut_4)=1$, then
\begin{align}
\P{\Phi}
=\frac{8}{3}+\frac{1}{c} +2-2k +\sum_{i=5}^{\nnuts}\P{Z}(\nut_i)
\leq -\frac{4}{3}+\frac{1}{c}
<0
\end{align} 
which is a contradiction, while if $\epsilon(\nut_4)=-1$, then 
\begin{align}
\M{\Phi}
=-1+\frac{1}{c} +2-2k +\sum_{i=5}^{\nnuts}\M{Z}(\nut_i)
\leq-1+\frac{1}{c}
\leq 0
\end{align}
and we conclude that $\M{\Phi} = 0$.

Finally, when $(a,b)=(2,3)$, we have the nut configuration 
\begin{align} 
\{+,3,5\},\quad \{+,2,5\}, \quad \{+,1,3\}, 
\end{align} 
and as in the previous case, there is a nut $\nut_4$ with data $\{\epsilon(\nut_4),2,c\}$. Reasoning as in the previous case, we find that $\P{\Phi} < 0$ or $\M{\Phi} < 0$ holds, depending on whether $\epsilon(\nut_4)=1$ or $\epsilon(\nut_4)=-1$.

\subsection*{Case: $\{+,b,a+b\}, \{+,a,a+b\},\{+,b,2b-a\}$}
In this case, 
\begin{align}
\P{\Phi}
=\frac{1}{a}+\frac{2}{b}+\frac{1}{2b-a} +2-2k +\sum_{i=4}^{\nnuts}\P{Z}(\nut_i)
\leq\frac{1}{a}+\frac{2}{b}+\frac{1}{2b-a}-2,
\end{align} 
and the right hand side is negative for $b>2$, which rules out this case.

For the case $(a,b)=(1,2)$, we have the nut configuration 
\begin{align} 
\{+,2,3\}, \quad \{+,1,3\}, \quad \{+,2,3\}, 
\end{align} 
and since the weight $3$ occurs an odd number of times, Lemma~\ref{lem:2.6} tells us that there exists another nut, say $\nut_4$, with nut data $\{\epsilon(\nut_4),3,c\}$, for some $c$. If $\epsilon(\nut_4)=1$, then
\begin{align}
\P{\Phi}
=
\frac{8}{3}+\frac{1}{c} +2-2k +\sum_{i=5}^{\nnuts}\P{Z}(\nut_i)
\leq -\frac{4}{3}+\frac{1}{c}
< 0,
\end{align} 
while if $\epsilon(\nut_4)=-1$, then 
\begin{align}
\M{\Phi}
=
-\frac{4}{3}+\frac{1}{c} +2-2k +\sum_{i=5}^{\nnuts}\M{Z}(\nut_i)
\leq -\frac{4}{3}+\frac{1}{c}
< 0,
\end{align}
which leads to a contradiction.

\subsection*{Case: $\{+,b,a+b\},\{+,a,a+b\},\{-,a,b\}$}
In this case, 
\begin{align}
\M{\Phi}
=\sum_{i=4}^{\nnuts}\M{Z}(\nut_i)+2-2k
\leq 0,
\end{align}
so $\M{\Phi} = 0$.

\subsection*{Case: $\{-,b,a+b\},\{-,a,a+b\},\{-,b-a,b\}$}
In this case, 
\begin{align}
\M{\Phi}
=
\frac{1}{a}+\frac{2}{b}+\frac{1}{b-a} +2-2k +\sum_{i=4}^{\nnuts}\M{Z}(\nut_i)
\leq \frac{1}{a}+\frac{2}{b}+\frac{1}{b-a}-4
< 0,
\end{align}
which is impossible.

\subsection*{Case: $\{-,b,a+b\},\{-,a,a+b\},\{-,b,2b-a\}$}
In this case, 
\begin{align}
\M{\Phi}
= \frac{1}{a}+\frac{2}{b}+\frac{1}{2b-a}
+2-2k+\sum_{i=4}^{\nnuts}\M{Z}(\nut_i)
\leq \frac{1}{a}+\frac{2}{b}+\frac{1}{2b-a}-4
<0,
\end{align}
which is a contradiction.

\subsection*{Case: $\{-,b,a+b\},\{-,a,a+b\},\{+,a,b\}$}
In this case, 
\begin{align}
\M{\Phi}
= \frac{2}{b} +2-2k +\sum_{i=4}^{\nnuts}\M{Z}(\nut_i)
\leq \frac{2}{b}-2
<0,
\end{align}
which again is a contradiction.
\end{proof}

\subsection{Proof of Theorem \ref{thm:main-intro}.} 
First assume that $(\cM, g_{ab})$ is an ALF $S^1$-instanton on $S^4 \setminus S^1$. Then, by Lemma \ref{lem:Kerr-eq}, we have that equality holds in \eqref{eq:magdim} for both signs. By Corollary \ref{cor:magdim}, it follows that $(\cM, g_{ab})$ is Petrov type D and hence toric \cite{2013arXiv1302.6975A}. The classification of \cite{Biquard:Gauduchon} now applies to show that $(\cM, g_{ab})$ is in the Kerr family. Similarly, if $(\cM, g_{ab})$ is an ALF $S^1$-instanton on $\CP^2 \setminus \{\text{\rm pt.}\}$, we have by Lemma \ref{lem:Tb-eq} that it is Petrov type D and by a similar argument toric and hence by \cite{Biquard:Gauduchon}, $(\cM, g_{ab})$ is in the Taub-bolt family. Finally, Lemma \ref{lem:CT-eq} gives point \ref{point:3} of Theorem \ref{thm:main-intro}, in particular an AF $S^1$-instanton on $\CP^2\setminus S^1$ is one-sided type D and therefore Hermitian. 

\section{Concluding remarks} 

In this paper we have applied the divergence identity derived in Section \ref{sec:divident} to the case of ALF $S^1$-instantons, restricting to those topologies where there are known Hermitian examples. However, it is worth emphasizing that the method can be applied more broadly. The $G$-signature  can, together with the approach of Jang, be used to classify $S^1$-actions on gravitational instantons. This allows one to apply the divergence identity to general $S^1$-instantons. These applications will be considered in forthcoming work. 

\subsection*{Acknowledgements}
W.S. acknowledges funding by the Austrian Science Fund (FWF) [Grant DOI 10.55776/P35078]. We are grateful to Claude LeBrun for helpful remarks, and to Olivier Biquard for several helpful discussions and for sharing an early version of \cite{Biquard:Gauduchon}. We thank the anonymous referees for valuable suggestions to improve the article.

\appendix 
\section{Proof of Proposition \ref{prop:xi-id}} \label{sec:xi-id-proof} 

First we recall the form of partly contracted volume forms,
\begin{align}
\epsilon_{abcd}\epsilon^{d}{}_{efh}
={}&g_{ah} g_{bf} g_{ce}
 -  g_{af} g_{bh} g_{ce}  -  g_{ah} g_{be} g_{cf}\nonumber \\
 & + g_{ae} g_{bh} g_{cf}
 + g_{af} g_{be} g_{ch}
 -  g_{ae} g_{bf} g_{ch}, \label{epseps1}\\
\epsilon_{abcd}\epsilon^{cd}{}_{fh}={}&-2 g_{ah} g_{bf}
 + 2 g_{af} g_{bh}.\label{epseps2}
\end{align}
By \eqref{epseps1}, the SD/ASD field $\tfrac{1}{2} \epsilon_{ab}{}^{cd}  \PM{\Fcal}_{cd} = \pm \PM{\Fcal}_{ab} $ satisfies
\begin{align} \label{eq:FcalFcal1contr}
\PM{\Fcal}_{ab} \PM{\Fcal}^{b}{}_{c} 
={}& \tfrac{1}{4} \epsilon_{ab}{}^{ef}\epsilon^{b}{}_{c}{}^{hi} \PM{\Fcal}_{ef} \PM{\Fcal}^{hi} \nonumber \\
={}& -\tfrac{1}{4}  \PM{\Fcal}^2 g_{ac}.
\end{align}
By \eqref{eq:FcalFcal1contr},
\begin{align}
\PM{\Fcal}_{ab} \PM{\sigma}^{b}
={}& 2 \PM{\Fcal}_{ab} \PM{\Fcal}^{b}{}_{c} \xi^{c} \nonumber \\
={}& - \tfrac{1}{2} \PM{\Fcal}^2\xi_{a}, 
\end{align}
which is \eqref{eq:Fsigma} and similarly, \eqref{eq:sigmasq} follows via
\begin{align}
\PM{\sigma}_{a} \PM{\sigma}^{a}
={}& 4 \PM{\Fcal}_{ab} \PM{\Fcal}^{a}{}_{c} \xi^{b} \xi^{c} \nonumber \\
={}& \PM{\Fcal}^2 \lambda.
\end{align}
By Ricci flatness, we have $\nabla_{b} F_{ac} = \nabla_{b}\nabla_{a}\xi_{c} = - W_{acbd} \xi^{d}$, see \cite[Proposition~8.1.3]{PetersenRiemannianGeometry3rd}, and hence \eqref{eq:nablaFcal} follows. This also leads to \eqref{eq:nablaFsq}\ via
\begin{align}
\nabla_{c}\PM{\Fcal}^2
={}&2 \PM{\Fcal}^{ab}\nabla_{c}\PM{\Fcal}_{ab} \nonumber \\
={}&- 2 \PM{\Fcal}^{ab}\PM{\Wcal}_{abcd} \xi^{d}.
\end{align}
To prove \eqref{eq:dsigmapm}, it is convenient to start with its dual,
\begin{align}
\epsilon_{fh}{}^{ab} \nabla_{a}\PM{\sigma}_{b} 
={}& 2\epsilon_{fh}{}^{ab} \nabla_{a}\left( \PM{\Fcal}_{bc}\xi^c \right) \nonumber \\
={}& \underbrace{ -2 \epsilon_{fh}{}^{ab}\PM{\Wcal}_{bcad}\xi^c \xi^d}_{=0} + 2\epsilon_{fh}{}^{ab}\PM{\Fcal}_{bc} F_a{}^c \nonumber \\
={}&  \underbrace{ 2 \epsilon_{fh}{}^{ab}  F_{bc} F_a{}^c}_{=0}  \pm \epsilon_{fh}{}^{ab}  \epsilon_{bc}{}^{ij} F_{ij}  F_a{}^c \nonumber \\
={}& 0.
\end{align}
The last step used \eqref{epseps1} and dualizing the equation shows \eqref{eq:dsigmapm}. The divergence, \eqref{eq:divsigma}, follows from
\begin{align}
\nabla_{a}\PM{\sigma}^{a}
={}& 2 \nabla_{a} \left( \PM{\Fcal}^{ab}\xi_{b}\right)\nonumber \\
={}& \underbrace{-2\PM{\Wcal}^{ab}{}_{ad}\xi_{b}\xi^d}_{=0} +  2\PM{\Fcal}^{ab} F_{ab} 
\nonumber \\
={}& \PM{\Fcal}^{ab} \left( \PM{\Fcal}_{ab} + \MP{\Fcal}_{ab} \right) \nonumber \\
={}&\PM{\Fcal}^2.
\end{align}
Finally, to prove \eqref{eq:laplaceFsq}, apply a derivative to \eqref{eq:nablaFsq},
\begin{align}
\nabla^{c}\nabla_{c}\PM{\Fcal}^2
={}&-2 \nabla^{c} \left( \PM{\Wcal}_{abcd} \xi^{d}\PM{\Fcal}^{ab} \right) \nonumber \\
={}& -2 \big( \underbrace{(\nabla^{c} \PM{\Wcal}_{abcd})}_{=0} \xi^{d}\PM{\Fcal}^{ab} +\PM{\Wcal}_{abcd} \PM{\Fcal}^{ab} F^{cd}  - \underbrace{\PM{\Wcal}_{abcd}  \PM{\Wcal}^{abc}{}_{e}}_{= \tfrac{1}{4} g_{de}\PM{\Wcal}^2} \xi^{d} \xi^e \big) \nonumber \\
={}& - \PM{\Wcal}_{abcd} \PM{\Fcal}^{ab}  \PM{\Fcal}^{cd} - \underbrace{\PM{\Wcal}_{abcd} \PM{\Fcal}^{ab}  \MP{\Fcal}^{cd}}_{=0} +\tfrac{1}{2} \lambda \PM{\Wcal}^2.
\end{align}
In the second line, we used the Bianchi identity for the first term and a decomposition of contracted Weyl tensors, analogous to \eqref{eq:FcalFcal1contr}, for the last term. The second term in the last line vanishes because it is SD and ASD in the index pair $cd$.

\section{Asymptotics at infinity} \label{sec:asympinf}

In this appendix, we prove a couple of technical results on the asymptotic geometry of ALF $S^1$-spaces.  

\subsection{Asymptotics of Killing vector fields} \label{sec:asympKill}
We will prove that a bounded Killing field with closed orbits for a general ALF metric $g$ must be asymptotic to a multiple of the Killing field $T$ of the background metric $\gring$. For a similar result, see \cite[Proposition~2.1]{Beig-Chrusciel-96}. We first formulate a lemma which is \cite[Lemma,~Appendix~A]{Chrusciel-88}, adapted to our setting.

\begin{lemma} \label{lemma-piotr-limit}
Let $f$ be a function on $ (A,\infty) \times L$ with $|df| = O(r^{-1-\alpha})$, $\alpha >0$. Then there is a constant $f_\infty$ so that $f = f_\infty + O(r^{-\alpha})$. 
\end{lemma}
\begin{prop} \label{prop-asymptotics-killing}
Let $(\cM, g_{ab})$ be an ALF space. Assume $\xi^a$ is a Killing field with uniformly bounded norm, all of whose orbits are closed. 
Then  
\begin{align} \label{eq:xiO*}
\xi^a = cT^a + O^*(r^{-1})
\end{align} 
for some constant $c$ and $T$ has closed orbits in $L$. In particular, 
\begin{align} \label{eq:nablaxiO*}
\nabla\xi = O^*(r^{-2}).
\end{align}  
\end{prop}

\begin{proof}
Since $\Lie_\xi g = 0$ and $\xi$ is bounded we have 
\begin{equation} \label{nabla-2-xi=riem-xi}
\nabla^2 \xi = \Riem \xi = O(r^{-3}),
\end{equation}
see \cite[Proposition~8.1.3]{PetersenRiemannianGeometry3rd}. Using the Kato inequality we find that 
\begin{align} 
|d |\nabla \xi| | \leq |\nabla^2 \xi| = |\Riem \xi| = O(r^{-3}).
\end{align} 
By Lemma~\ref{lemma-piotr-limit} there is a constant $w_\infty$ so that
\begin{align} 
|\nabla \xi| = w_\infty + O(r^{-2}).
\end{align} 
and in particular $|\nabla \xi|$ is bounded.

We next work in local coordinates as discussed in Remark \ref{rem:cone} on 
\begin{align}
C \times (-t_0,t_0) = (A, \infty) \times U \times (-t_0,t_0)
\subset \Reals^3 \times (-t_0,t_0) 
\end{align} 
for $U$ an open subset of $S^2$. We denote the standard coordinates on $\Reals^3$ by $(x_1,x_2,x_3)$ and the coordinate on $(-t_0,t_0)$ by $x_0=t$. In these coordinates we have $g = \delta + dt^2 + O(r^{-1})$. Then 
\begin{equation} \label{nabla-xi-local}
\nabla \xi = d \xi + \Gamma \xi = W + O(r^{-2})
\end{equation}
where $\Gamma = O(r^{-2})$ represents the Christoffel symbols for $g$, and $W$ is a $4\times 4$-matrix of functions representing $d\xi$. From \eqref{nabla-2-xi=riem-xi} we find that
\begin{equation} \label{dW-equation}
dW = \nabla^2 \xi - \Gamma \nabla \xi = \Riem \xi - \Gamma \nabla \xi.
\end{equation}
Since $\nabla \xi$ is bounded we find that $|dW| = O(r^{-2})$ and from Lemma~\ref{lemma-piotr-limit} we conclude that $W = W^\infty + O(r^{-1})$ for a constant $4\times 4$-matrix $W^\infty$. We thus get 
\begin{align}
d\xi = W^\infty + O(r^{-1}). 
\end{align}
Since $\xi$ is a Killing vector field for $g$ we have
\begin{align}
0 = g(\nabla_{\partial_i} \xi, \partial_j) + g(\partial_i, \nabla_{\partial_j} \xi) 
= W^\infty_{ij} + W^\infty_{ji} + O(r^{-1})
\end{align} 
for $0\leq i,j \leq 3$. We conclude that $W^\infty$ is skew-symmetric.

Let the point $p_0 \in C$ be such that the points $p_0 + e_i$, $i=1,2,3$ are also in $C$. Then since $C$ is a cone the line segments $R(p_0 + s e_i)$, $0 \leq s \leq 1$, are also contained in $C$ for $R \geq 1$. Let $c_i^R$ be the curve 
\begin{align}
c_i^R(s) = (R(p_0 + s e_i),0 ) \subset C \times (-t_0,t_0), \qquad 0\leq s \leq 1.
\end{align} 
Then 
\begin{align} 
\xi(c_i^R(1)) - \xi(c_i^R(0)) 
={}& \int_0^1 d\xi(R e_i) \, ds \nonumber\\
={}& R\int_0^1 d\xi(e_i) \, ds \nonumber \\
={}& R\int_0^1 \left( W^\infty + O(r^{-1}) \right)(e_i) \, ds \nonumber \\
={}& R (W^\infty_{ij} e_j + O(R^{-1})) \nonumber \\
={}& R W^\infty_{ij} e_j + O(1).
\end{align} 
Since $\xi$ is bounded as $R \to \infty$ we find that $W^\infty_{ij} = 0$ for $i=1,2,3$. Since $W^\infty$ is skew-symmetric we conclude that $W^\infty = 0$. 

Inserting $W = W^\infty + O(r^{-1}) = O(r^{-1})$ into \eqref{nabla-xi-local} we get that $\nabla \xi = O(r^{-1})$, which we in turn insert into \eqref{dW-equation} to see that $|dW| = O(r^{-3})$. From Lemma~\ref{lemma-piotr-limit} we conclude that 
\begin{align}
W = W^\infty + O(r^{-2}) = O(r^{-2})
\end{align}
and 
\begin{align}
\nabla \xi = O(r^{-2}).
\end{align} 

We next write $\xi = V$ in coordinates, where $V$ is a vector of functions. From 
\begin{align} 
\nabla \xi = dV + \Gamma V  
\end{align} 
we see that $|dV| = O(r^{-2})$, so from Lemma~\ref{lemma-piotr-limit} we get $V = V^\infty + O(r^{-1})$ for a constant vector $V^\infty$. 

In the local coordinate system we have found that
\begin{align}\label{eq:nablaxiO} 
\xi ={}&  \sum_{i=0}^3 V^\infty_i e_i + O(r^{-1}) \nonumber \\
={}&  V^\infty_0 \partial_t + \sum_{i=1}^3 V^\infty_i e_i + O(r^{-1}) \nonumber \\
={}&  V^\infty_0 T + \sum_{i=1}^3 V^\infty_i e_i + O(r^{-1}).
\end{align} 
By continuity the coefficents $V^\infty_0 = c$ must be independent of the choice of local coordinates. Patching together with a partition of unity we find
\begin{align} 
\xi = cT + \eta + O(r^{-1})
\end{align} 
where $\eta = Z \partial_r + \zeta$ is a vector field on $\Mring = (A,\infty)\times L$ with $Z$ a constant and $\zeta$ a ($r$-independent) vector field on $L$, and $\gring(T,\eta) = 0$. 

By assumption, the orbits of $\xi$ are all closed. Since $\xi$ is bounded the orbits have uniformly bounded lengths. If now $\eta\neq 0$ it is easy to see that $\xi$ can have arbitrarily long orbits by choosing a starting point with sufficiently large $r$. We conclude that $\eta=0$ and $\xi = cT + O(r^{-1})$. 
Using the above estimates together with $\Riem = O^*(r^{-3})$ and \eqref{nabla-2-xi=riem-xi} we can now conclude that in fact \eqref{eq:xiO*} and \eqref{eq:nablaxiO*} hold. 

It remains to show that $T$ has closed orbits on $L$. To see this, let $p \in L$ be a point whose $T$-orbit is not closed, and consider a sequence of points $p_i = (r_i, p) \in \mathring{\cM}$, for $r_i \to \infty$. The $\xi$-orbits of $p_i$ are closed with uniformly bounded lengths $\ell_i$. Passing to a subsequence, we may assume that $\ell_i$ tends to a limit $\ell_\infty$. Flowing $p$ along $T$ a time $\ell_\infty$ yields a point $q$ with $d(p, q) > 0$. This contradicts \eqref{eq:xiO*}. 
\end{proof}

\subsection{Asymptotics of solutions to Poisson equations.} \label{sec:asympPoi}  
Let $(\cM, g_{ab})$ be an ALF $S^1$ instanton. In this section, we shall analyze the asymptotics at infinity for solutions of the Poisson equation of the form 
\begin{align} 
\Delta u = O^*(r^{-4}), \quad \text{with $u = O^*(r^{-1})$}.
\end{align} 
First we prove a lemma for almost invariant functions on $(\Mring,\gring)$.  

\begin{lemma} \label{lemma-poisson-gring} 
Let $(\mathring{\cM}, \mathring{g}_{ab})$, $T$ be as in Definition \ref{def:BG}. Suppose $u \in C^2_{loc}(\cM)$, $u = O^*(r^{-1})$ satisfies $Tu = O^*(r^{-3})$ and 
\begin{align}\label{eq:Pois*} 
\Delta^{\gring}u = O^*(r^{-4}).
\end{align}
Then there are constants $b, \delta \in \Reals$ so that 
\begin{align} 
u = br^{-1} + O(r^{-1-\delta}).
\end{align} 
\end{lemma}

\begin{proof}
We have $\gring = dr^2 + \sigma_r$ where $\sigma_r = r^2\gamma + \eta^2$. Let $\sigma_1 = \gamma + \eta^2$.
Then
\begin{align}
\Delta^{\gring} u 
= r^{-2} \partial_r (r^2 \partial_r u) + \Delta^{\sigma_r} u. 
\end{align}
Since $T$ is a Killing vector field for $\gring$ we get 
\begin{align}
r^2 \Delta^{\sigma_r} u = \Delta^{\sigma_1}u + (r^2-1) T^2 u.
\end{align}
We thus have
\begin{align}
r^{-2} \partial_r (r^2 \partial_r u) + r^{-2} \Delta^{\sigma_1} u
+r^{-2} (r^2-1) T^2 u
= O(r^{-4})
\end{align}
or
\begin{align}
r^2 \partial_r^2 u + 2r \partial_r u + \Delta^{\sigma_1} u + (r^2-1) T^2 u
= O(r^{-2}).
\end{align}
By assumption $Tu = O^*(r^{-3})$, and hence $T^2u = O^*(r^{-4})$ and $(r^2-1) T^2 u = O^*(r^{-2})$. This leads to an equation of the form 
\begin{align} \label{eq:ueq}
r^2 \partial_r^2 u + 2r \partial_r u + \Delta^{\sigma_1} u = f 
\end{align}
for some $f$ satisfying
\begin{align}
f = O^*(r^{-2}).
\end{align}

Let $D = \sqrt{1 - 4\Delta^{\sigma_1}}$ as a first order pseudodifferential operator. Then the equation factorizes as
\begin{align}
r^{\frac{1}{2}(1+D)} \partial_r \left( 
r^{1-D} \partial_r \left( r^{\frac{1}{2}(1+D)} u \right)
\right) = f,
\end{align}
or
\begin{align}
\partial_r \left( 
r^{1-D} \partial_r \left( r^{\frac{1}{2}(1+D)} u \right)
\right) = r^{-\frac{1}{2}(1+D)} f.
\end{align}

Since $D \geq 1$ we have
\begin{align}
\left\| r^{-\frac{1}{2} (D-1) } \right\|_{\operatorname{Op}(H_s \to H_s)}
\leq 1
\end{align}
for $r \geq 1$. Since $u = O^*(r^{-1})$ we get
\begin{align}
\left\| 
r^{-\frac{1}{2} (D-1) } \partial_r u 
\right \|_{L^\infty}
&\leq 
C \left\| 
r^{-\frac{1}{2} (D-1) } \partial_r u 
\right \|_{H_s} \nonumber \\ 
&\leq 
C \left\| 
\partial_r u 
\right \|_{H_s} \nonumber \\
&= O(r^{-2}) 
\end{align}
and
\begin{align}
\left\| 
r^{-\frac{1}{2} (D-1) } \left( \frac{1}{2}(1+D) u \right)
\right \|_{L^\infty}
&\leq 
C \left\| 
r^{-\frac{1}{2} (D-1) } \left( \frac{1}{2}(1+D) u \right)
\right \|_{H_s} \nonumber \\ 
&\leq 
C \left\| 
\left( \frac{1}{2}(1+D) u \right)
\right \|_{H_s} \nonumber \\ 
&\leq 
C \left\| u \right \|_{H_{s+1}} \nonumber \\
&= O(r^{-1}),
\end{align}
so we find that  
\begin{align} 
r^{1-D} \partial_r \left( r^{\frac{1}{2}(1+D)} u \right) 
&= r^{1 -\frac{1}{2} (D-1) } \partial_r u
+ r^{-\frac{1}{2} (D-1) } \left( \frac{1}{2}(1+D) u \right) \nonumber \\
&= r O(r^{-2}) + O(r^{-1}) \nonumber \\
&= O(r^{-1}).
\end{align} 
We may thus integrate from $r$ to $\infty$ to conclude that
\begin{align}
0 - r^{1-D} \partial_r \left( r^{\frac{1}{2}(1+D)} u \right)
= \int_r^\infty s^{-\frac{1}{2}(1+D)} f(s) \, ds
\end{align}
or
\begin{align}
\partial_r \left( r^{\frac{1}{2}(1+D)} u \right)
= - r^{D-1} \int_r^\infty s^{-\frac{1}{2}(1+D)} f(s) \, ds .
\end{align}

Let
\begin{align}
1 = \delta_0 < \delta_1 \leq \dots \leq \delta_k \leq 3
\end{align}
be the eigenvalues of $D$ with value less than or equal to $3$ and let $\phi_0, \dots, \phi_k$ be the corresponding orthonormal eigenfunctions. Set
\begin{align}
u_i(r) = \int_L u \phi_i \, d\mu^{\sigma_1},
\qquad
f_i(r) = \int_L f \phi_i \, d\mu^{\sigma_1}.
\end{align}
Then
\begin{align}
\partial_r \left( r^{\frac{1}{2}(1+\delta_i)} u_i(r) \right)
= - r^{\delta_i-1} \int_r^\infty s^{-\frac{1}{2}(1+\delta_i)} f_i(s) \, ds .
\end{align}

If $\delta_i < 3$ we integrate from $r$ to $q$, $r < q$, and get
\begin{align}
q^{\frac{1}{2}(1+\delta_i)} u_i (q)
-r^{\frac{1}{2}(1+\delta_i)} u_i (r)
= -
\int_r^q
t^{\delta_i-1} \int_t^\infty s^{-\frac{1}{2}(1+\delta_i)} f_i(s) \, ds \, dt
\end{align}
so
\begin{align}
\left| q^{\frac{1}{2}(1+\delta_i)} u_i (q)
-r^{\frac{1}{2}(1+\delta_i)} u_i (r) \right|
&\leq
\int_r^q
t^{\delta_i-1} \int_t^\infty s^{-\frac{1}{2}(1+\delta_i)} |f_i(s)|
\, ds \, dt \nonumber \\
&\leq
C\int_r^q
t^{\delta_i-1} \int_t^\infty s^{-\frac{1}{2}(1+\delta_i)} s^{-2}
\, ds \, dt \nonumber \\
&\leq
C\int_r^q
t^{\delta_i-1} t^{-\frac{1}{2}(3+\delta_i)} \, dt \nonumber \\
&=
C\int_r^q
t^{-\frac{1}{2}(5-\delta_i)} \, dt \nonumber \\
&\leq
C \left( r^{-\frac{1}{2}(3-\delta_i)}  - q^{-\frac{1}{2}(3-\delta_i)}\right).
\end{align}
and we conclude that the function
$q \mapsto q^{\frac{1}{2}(1+\beta)} u_i(q)$ is bounded and the limit
\begin{align}
\lim_{q \to \infty} q^{\frac{1}{2}(1+\delta_i)} u_i (q)
= A_i
\end{align}
exists. Taking the limit $q \to \infty$ of the above we get
\begin{align}
\left| A_i - r^{\frac{1}{2}(1+\delta_i)} u_i (r) \right|
\leq C r^{-\frac{1}{2}(3-\delta_i)}
\end{align}
or
\begin{align}
u_i (r) = r^{-\frac{1}{2}(1+\delta_i)} A_i + O(r^{-2}).
\end{align}

If $\delta_i = 3$ we integrate from $1$ to $r$ to get 
\begin{align}
r^2 u_i (r)
- u_i (1)
= -
\int_1^r t^{2} \int_t^\infty s^{-2} f_i(s) \, ds \, dt
\end{align}
so
\begin{align}
\left | r^2 u_i (r)
- u_i (1) \right|
&\leq
\int_1^r t^{2} \int_t^\infty s^{-2} | f_i(s) | \, ds \, dt \nonumber \\
&\leq
C\int_1^r t^{2} \int_t^\infty s^{-2} s^{-2} \, ds \, dt \nonumber \\
&\leq 
C\int_1^r t^{2} t^{-3} \, dt \nonumber \\
&= 
C \ln r
\end{align}
and
\begin{align}
u_i (r)
= r^{-2} u_i(1) + O(r^{-2} \ln r) 
= O(r^{-2} \ln r). 
\end{align}

Let $\delta$ be the smallest eigenvalue of $D$ which is $> 3$ and let $P$ be the spectral projection for $D$ corresponding to eigenvalues $\geq \delta$. Let $U = P u$ and $F = P f$. Then
\begin{align}
\partial_r \left( r^{\frac{1}{2}(1+D)} U \right)
= - r^{D-1} \int_r^\infty s^{-\frac{1}{2}(1+D)} F(s) \, ds 
\end{align}
which we integrate from $1$ to $r$ to get
\begin{align}
r^{\frac{1}{2}(1+D)} U(r) - U(1)
= - 
\int_1^r
t^{D-1} \int_t^\infty s^{-\frac{1}{2}(1+D)} F(s) \, ds 
\, dt
\end{align}
or
\begin{align}
U(r) 
&= r^{- \frac{1}{2}(1+D)}  U(1) 
-  r^{-\frac{1}{2}(1+D)} \int_1^r
t^{D-1} \int_t^\infty s^{-\frac{1}{2}(1+D)} F(s) \, ds 
\, dt \nonumber \\
&= K(r) + I(r).
\end{align}
Since $D - \delta \geq 0$ on the image of $P$ we have
\begin{align}
\left\| r^{-\frac{1}{2} (D - \delta) }
\right\|_{\operatorname{Op}(H_s \cap \operatorname{Im}P
\to H_s \cap \operatorname{Im}P)}
\leq 1
\end{align}
for $r \geq 1$. Here $\operatorname{Op}$ is the operator norm. For $K(r)$ we have 
\begin{align}
K(r) 
= 
r^{- \frac{1}{2}(1+D)}  U(1)  
=
r^{- \frac{1}{2}(\delta + 1)} r^{- \frac{1}{2}(D - \delta)}  U(1)
\end{align}
so
\begin{align}
\left\| K(r) \right\|_{L^\infty} 
&\leq 
C r^{- \frac{1}{2}(\delta + 1)}  
\left\| r^{- \frac{1}{2}(D - \delta)}  U(1) \right\|_{H_s} \nonumber \\
&\leq 
C r^{- \frac{1}{2}(\delta + 1)}
\left\| U(1) \right\|_{H_s} \nonumber \\
&\leq
C r^{- \frac{1}{2}(\delta + 1)}.
\end{align}
For $I(r)$ we have
\begin{align}
I(r)
&=
-r^{-\frac{1}{2}(1+D)} \int_1^r
t^{D-1} \int_t^\infty s^{-\frac{1}{2}(1+D)} F(s) \, ds \, dt \nonumber \\
&= 
- \int_1^r \int_t^\infty 
r^{-\frac{1}{2}(1+D)}
t^{D-1}
s^{-\frac{1}{2}(1+D)} F(s) \, ds \, dt \nonumber \\
&=
- \int_1^r \int_t^\infty 
r^{-\frac{1}{2}(D -\delta + \delta +1)}
t^{D-\delta + \delta-1}
s^{-\frac{1}{2}(D-\delta + \delta+1)} F(s) 
\, ds \, dt \nonumber \\
&=
- \int_1^r \int_t^\infty 
r^{-\frac{1}{2}(\delta +1)}
t^{\delta-1}
s^{-\frac{1}{2}(\delta+1)} F(s) 
r^{-\frac{1}{2}(D -\delta)}
t^{D-\delta}
s^{-\frac{1}{2}(D-\delta)} F(s) 
\, ds \, dt \nonumber \\
&=
- r^{-\frac{1}{2}(\delta +1)} \int_1^r 
t^{\delta-1} \int_t^\infty 
s^{-\frac{1}{2}(\delta+1)} 
\left( \frac{r}{t} \right)^{-\frac{1}{2}(D -\delta)}
\left( \frac{s}{t} \right)^{-\frac{1}{2}(D-\delta)} F(s) 
\, ds \, dt .
\end{align}
Since $s \geq t$ and $t \leq r$ we have
\begin{align}
\left\| \left(\frac{s}{t} \right)^{-\frac{1}{2} (D - \delta) }
\right\|_{\operatorname{Op}(H_s \cap \operatorname{Im}P
\to H_s \cap \operatorname{Im}P)}
\leq{}& 1, \\ 
\left\| \left(\frac{r}{t} \right)^{-\frac{1}{2} (D - \delta) }
\right\|_{\operatorname{Op}(H_s \cap \operatorname{Im}P
\to H_s \cap \operatorname{Im}P)}
\leq{}& 1,
\end{align}
and since $F(r) = O(r^{-2})$,
\begin{align}
\left\| I(r) \right\|_{L^\infty}
&\leq
r^{-\frac{1}{2}(\delta +1)} \int_1^r 
t^{\delta-1} \int_t^\infty 
s^{-\frac{1}{2}(\delta+1)} 
\left\|
\left( \frac{r}{t} \right)^{-\frac{1}{2}(D -\delta)}
\left( \frac{s}{t} \right)^{-\frac{1}{2}(D-\delta)} F(s) 
\right\|_{L^\infty}
\, ds \, dt \nonumber \\
&\leq
Cr^{-\frac{1}{2}(\delta +1)} \int_1^r 
t^{\delta-1} \int_t^\infty 
s^{-\frac{1}{2}(\delta+1)} 
\left\|
\left( \frac{r}{t} \right)^{-\frac{1}{2}(D -\delta)}
\left( \frac{s}{t} \right)^{-\frac{1}{2}(D-\delta)} F(s) 
\right\|_{H_s}
\, ds \, dt \nonumber \\
&\leq
Cr^{-\frac{1}{2}(\delta +1)} \int_1^r 
t^{\delta-1} \int_t^\infty 
s^{-\frac{1}{2}(\delta+1)} 
\left\| F(s) \right\|_{H_s}
\, ds \, dt \nonumber \\
&\leq
Cr^{-\frac{1}{2}(\delta +1)} \int_1^r 
t^{\delta-1} \int_t^\infty 
s^{-\frac{1}{2}(\delta+1)} s^{-2}
\, ds \, dt \nonumber \\
&=
Cr^{-\frac{1}{2}(\delta +1)} \int_1^r 
t^{\delta-1} \int_t^\infty 
s^{-\frac{1}{2}(\delta+5)}
\, ds \, dt \nonumber \\
&=
Cr^{-\frac{1}{2}(\delta +1)} \int_1^r 
t^{\delta-1} t^{-\frac{1}{2}(\delta+3)}
\, dt \nonumber \\
&=
Cr^{-\frac{1}{2}(\delta +1)} \int_1^r 
t^{\frac{1}{2}(\delta-5)}
\, dt \nonumber \\
&=
Cr^{-\frac{1}{2}(\delta +1)} 
\left( r^{\frac{1}{2}(\delta-3)} - 1 \right) \nonumber \\
&\leq 
Cr^{-2}. 
\end{align}
Together we have
\begin{align}
U(r) = O(r^{-2}).
\end{align}

Finally,
\begin{align}
u(r) &= 
u_0(r)\phi_0 
+ \sum_{0 < \delta_i < 3} u_i(r) \phi_i 
+ \sum_{\delta_i = 3} u_i(r) \phi_i 
+ U(r) \nonumber \\
&=
r^{-1} A_0 \phi_0
+ \sum_{0 < \delta_i < 3} 
\left( r^{-\frac{1}{2}(1+\delta_i)} A_i + O(r^{-2}) \right)\phi_i \nonumber \\
&\qquad
+ \sum_{\delta_i = 3} O(r^{-2} \ln r)  \phi_i
+ O(r^{-2})\nonumber \\
&= b r^{-1} + O(r^{-1-\delta})
\end{align}
for some $\delta > 0$.
\end{proof}

Next we consider the general case of $S^1$-invariant functions on an $S^1$-ALF manifold. 

\begin{prop} \label{prop:poisson-g}
Let $(M,g)$ be an $S^1$-ALF manifold where the circle action is generated by the bounded Killing field $\xi$. Suppose $u = O^*(r^{-1})$ is an $S^1$-invariant solution to
\begin{align}
\Delta^{g}u = O^*(r^{-4}).
\end{align}
Then there are constants $\delta >0$ and $b \in \Reals$ so that 
\begin{align} 
u = br^{-1} + O(r^{-1-\delta}).
\end{align} 
\end{prop}

\begin{proof}
Since $u = O(r^{-1})$ and $\xi u = 0$ it follows from Proposition~\ref{prop-asymptotics-killing} that $Tu = O(r^{-3})$. From $g = \gring + O(r^{-1})$ and the standard formula for the Laplacian
\begin{align}
\Delta^g u 
= \frac{1}{\sqrt{\det g}} \partial_i \left( \sqrt{\det g} g^{ij}  \partial_j u \right)
\end{align}
we see that $\Delta^{\gring}u = O^*(r^{-4})$. The result thus follows from Lemma~\ref{lemma-poisson-gring}.
\end{proof}

\section{The quotient} \label{sec:quotient}

\subsection{Generalities}
In this section we consider the orbit space $\cN = \cM/ S^1$, on which the family  of divergence identities was found originally 
\cite{MR747036,simon:1995}. We recall its algebraic derivation  from these papers; as to its  relation to the lifted family \eqref{eq:DivPsialpha} we make use of \cite{MR1701088}. 
However, before doing this we give a mathematically sound introduction of the orbit space.
The key is the following result

\begin{thm}[\cite{MR889050}] \label{thm:quot}
For a compact Lie group $G$ acting smoothly on $(\cM, g_{ab})$ the following holds.  
\begin{enumerate}
\item 
There exists a unique maximal orbit type.
\item 
The union $\cM_0$ of maximal orbits is open and dense in $\cM$.
\item 
The $G$- action on $\cM$ restricts to $\cM_0$, and  $\cM_0 \rightarrow \cM_0 /G $
is a Riemannian submersion. It is also a fibre bundle with fibre G/H where $H$ is a principal isotropy group.
\item 
The quotient $\cN_0= \cM_0 /G$ of the principal part is open, dense and connected in $\cM /G $.
\end{enumerate} 
\end{thm}

In the present case we  have $G = H = S^1$.  
Maximal orbits in the above sense  all have generic period in the sense of Lemma \ref{lem:generic}.  
Orbits which are not maximal are called exceptional.
While all orbits in a neighborhood of a bolt have generic period, there are exceptional orbits in the neighborhood of a non-self dual nut.

We now recall \cite{1971JMP....12..918G} that, for an $S^1$-isometry with Killing field $\xi^a$, there is a local diffeomorphism between 
\begin{itemize} 
\item
tensor fields and differential forms on $\cN_0$ and 
\item 
tensor fields and forms on $\cM_0$ which are orthogonal to $\xi^a$ and  commute with  $\xi^a$.
\end{itemize}
With a slight abuse of notation, we shall say that latter objects are also ``on  $\cN_{0}$'' (that is we drop the term ``push-forward''). For objects which are defined directly  on $\cN_0$ we use greek indices $(\al, \be...= 1,2,3)$. 

In particular, we obtain a non-degenerate metric $\gamma_{\mu\nu}$ ($\mu$, $\nu$ = 1,2,3) on $\cN_0$ via
\begin{equation}
\gamma_{ab} = \lambda g_{ab} - \xi_{a}\xi_{b}.
\label{2Gamma}
\end{equation}

We denote by $D_{\mu}$, $\cR_{\mu\nu}$ and $\cR$ the covariant derivative, the Ricci tensor and Ricci scalar with respect to $\gamma_{\mu\nu}$.

In local  coordinates $x^{a}=(t,x^{\mu})$ adapted to the isometry this reduction corresponds to a decomposition of $g_{ab}$ as
\begin{equation}
ds^{2} = \lambda(dt + \sigma_{\mu}dx^{\mu})^{2} + 
\lambda^{-1} \gamma_{\mu\nu} dx^{\mu} dx^{\nu} \label{2ds}
\end{equation}
where $\sigma_{\mu}$ is related to the pullbacks of $\omega_{a}$ and
$\epsilon_{abcd} \xi^{d}$ to $\cN_0$ by
\begin{equation}
\omega_{\mu} = - \lambda^{2} \epsilon_{\mu\nu\tau} D^{\nu} \sigma^{\tau}.
\label{omdes}
\end{equation}

In the coordinates introduced above, 1-forms $V_{a}$  on $\cM_0$ take the form $(0, V_{\mu})$ and the dual vector fields split as $V^{a} = (0, \lambda V^{\mu})$. The volume forms are related via $\sqrt{ \det g} = \lambda^{-1}  \sqrt{\det \gamma}$. For the divergence operators $\nabla_a$ and $D_{\mu}$ acting on vectors  $V^a$ on $\cN_0$ we obtain 
\begin{equation}
\nabla_{a} V^{a} = \lambda D_{\mu}  V^{\mu}.
\end{equation} 
From \eqref{omdes} it follows that
\begin{equation} \label{divom}
D_{\mu} (\lambda^{-2} \omega^{\mu}) = 0.
\end{equation}
which corresponds to the earlier observation that the current \eqref{eq:JT-def} is conserved.

Excluding the half-flat case we recall from \eqref{point:cE} of Lemma \ref{lem:Eprop} that  $|\PM{\cE}| < 1$ on $\cN_0$, and  we introduce $w$, $\Th$, $A_{\mu}$ via
\begin{align} \label{w}
\PM{w} ={}& (1+ \PM{\cE})^{-1}(1 -  \PM{\cE})  \\
\Theta ={}& 1 - \P{w}\M{w} =  4 \lambda (1 + \P{\cE})^{-1} (1 + \M{\cE})^{-1} \label{Th} \\
  A_{\mu} ={}& \frac{1}{2}(\P{w}D_{\mu}\M{w}-\M{w}D_{\mu}\P{w}) = \frac{(1 - \M{\cE})^2 \P{\sigma}_{\mu} - (1 - \cE_+^2) \M{\sigma}_{\mu}}
{(1 + \P{\cE})^2 (1 + \M{\cE})^2}
\label{A}
\end{align}

Rewriting \cite[Eq. (5.18)]{gibbons:hawking:1979}  we recall that the  field equations on $\cN_0$ can be derived by varying the Lagrangian
\begin{equation} \label{Lag}
L = \sqrt{\det \gamma}[\cR-2\Theta^{-2}\gamma^{\mu\nu}(D_{\mu}\P{w})(D_{\nu}\M{w})]
\end{equation}
with respect to $\PM{w}$ and $\gamma_{\mu\nu}$. This results in
\begin{align}
\label{lapw}
D_{\mu}\left( \Theta^{-1} \PM{\wh D}{}^{\mu} \PM{w} \right) ={}&  0 \\
\label{Ric}
\cR_{\mu\nu} ={}& 2\Theta^{-2}(D_{(\mu}\M{w})(D_{\nu)}\P{w}).
\end{align}
where we have introduced
\begin{equation}
\label{dhat}
\PM{\wh D}_{\mu} = D_{\mu} \mp 2\Theta^{-1}A_{\mu}.
\end{equation}
The $\PM{\wh D}{}^{\mu}$ is related to the operator $\mathcal{D}^{\mu}_{\pm}$ of \cite{simon:1995}
via $\mathcal {D}^{\mu}_{\pm}  = \Theta^{-1} \PM{\wh D}{}^{\mu}$.
  
As a consequence of \eqref{lapw} we also obtain 
\begin{equation}
D_{\mu}( \Theta^{-2} A^{\mu}) = 0. 
\end{equation}

\begin{remark}~
\begin{enumerate}
\item
As discussed in \cite[Sect. 5]{gibbons:hawking:1979}, equations \eqref{lapw} and \eqref{Ric} are invariant under the (Geroch-) $SL(2,\mathbb{R})$-group \cite{1971JMP....12..918G} which can be represented by the shift $\omega \rightarrow \omega + a$ with fixed $\lambda$, the rescalings $\PM{\cE} \rightarrow b \PM{\cE}$ and the Ehlers transform $\PM{w} \rightarrow c^{\pm 1} \PM{w}$ with constants $a$, $b$ and $c$, and with fixed $\gamma_{\mu\nu}$. The Noether currents corresponding to these symmetries on $\cM$ are given by \eqref{eq:currents}. 
\item
A special class of solutions is characterized by $\P{d} \P{w} = \M{d} \M{w}$ for some constants $\PM{d} \in {\mathbb{R}}$. This includes the cases in which one of $\PM{w}$ (and hence $\MP{d}$) vanishes identically, which have been called ``Multi-NUT'', see \cite[(3.19)]{gibbons:hawking:1979}. From \eqref{Ric}, $\gamma_{\mu\nu}$ is flat for these solutions and using \eqref{Th} and \eqref{A} we see that \eqref{lapw} reduces to an ordinary Laplace equation for the non-vanishing $\PM{w}$.  The generic solutions where $\P{d} \P{w} = \M{d} \M{w}$ but $\P{w} \not\equiv 0$ and $\M{w} \not\equiv 0$ can be transformed by an Ehlers transformation to the case $ \P{w} = \M{w}$ which means $\omega \equiv 0$ and hence corresponds to a hypersurface-orthogonal Killing field $\xi^{a}$ on $\cM_0$.
\end{enumerate}
\end{remark}

\subsection{The divergence identity on the quotient} \label{sec:divid}
We now recall the derivation of the family of  divergence identities on the quotient $\cN_0$. We first introduce  the pairs of quantities
\begin{definition}
\begin{align}   
\PM{k}^{4} ={}& D^{\mu} \PM{w} D_{\mu}\PM{w} \\
\PM{B}_{\mu\nu} ={}& 4\Theta^{-2}\mathcal{C}
\left(D_{\mu}D_{\nu}\PM{w} - (3(\PM{w})^{-1} + \Theta^{-1} \MP{w})D_{\mu}\PM{w} D_{\nu}\PM{w}\right) \\
\PM{C}_{\mu\nu\si} ={}& 4\Theta^{-2}
\left(D_{\mu}D_{[\nu}\PM{w}D_{\sigma]}\PM{w} - \gamma_{\mu[\nu} \PM{u}_{\sigma]}\right)
\end{align}
where 
\begin{equation} 
\PM{u}_{\si} = \gamma^{\mu\nu} D_{\mu} D_{[\nu} \PM{w} D_{\si]} \PM{w}
\end{equation}
and  $\mathcal{ C}$ denotes the trace-free part of a symmetric tensor.
\end{definition}
\begin{remark}
The quantities $\PM{k}$ and the push-forwards $C_{abc}$ of $C_{\mu\nu\nu}$ to ${\cM}_0$ are related to the ones defined in  \eqref{eq:munudef}, \eqref{eq:MSscalar} and \eqref{stst} via
\begin{align}
\PM{k}^4  ={}& \frac{16 (\mu \pm \nu)}{(1 + \PM{\cE})^4} = 4 \PM{w}^4 \PM{\sbold}^2 \\
\PM{C}_{abc}  ={}& -  \frac{(1 + \MP{\cE})^2}{8 \lambda^2 (1 + \PM{\cE})^2} \PM{\mathcal{P}}_{abc}
\end{align}
Furthermore, $\PM{B}_{\mu\nu}$ and $\PM{C}_{\mu\nu\si}$ are related via 
\begin{align}
\PM{C}_{\mu\nu\si} = \PM{B}_{\mu[\nu} D_{\si]}\PM{w} + \frac{1}{2} \gamma_{\mu[\nu} \PM{B}_{\si]\tau} D^{\tau} \PM{w}.
\end{align}
\end{remark}
We obtain the following identities \cite{simon:1995}
\begin{lemma}
\label{lem:divid}
On sets where $\PM{k}^{4} \ne 0$ \eqref{lapw} and \eqref{Ric} imply,
for each $\al \in \mathbb{R}$,
\begin{align}
\label{ddkw}
D_{\mu}\left(\Th^{-1}  \PM{\wh D}{}^{\mu} \frac{\PM{k}^{\al + 1}}{\PM{w}^\al} \right)
={}& \al (\al + 1) \frac{\PM{k}^{\al - 1}}
{\Theta \PM{w}^{\al}}
(D_{\mu} \PM{k} - \frac{\PM{k}}{\PM{w}}D_{\mu}\PM{w})
(D^{\mu} \PM{k} - \frac{\PM{k}}{\PM{w}}D^{\mu}\PM{w})
\nonumber \\
& +\frac{\al + 1}{16} \frac{\PM{k}^{\al - 7}}{\PM{w}^{\al}}
\Theta^{3} \PM{C}_{\mu\nu\si} \PM{C}^{\mu\nu\si}.
\end{align}
\end{lemma}

\begin{proof}
We  first sketch the straightforward calculation in the case $\al = 0$.
\begin{align} \label{b1} 
D_{\mu} \left( \Th^{-1}  \PM{\wh D}{}^{\mu} \PM{k} \right) 
={}& 2 \Th^{-2} \MP{w}D_{\mu} \PM{w} D^{\mu} \PM{k} + \Th^{-1} \Delta \PM{k}  \nonumber \\
={}& 2 \Th^{-2} \MP{w} D_{\mu} \PM{w} D^{\mu} \PM{k} -
\frac{3}{2} \Th^{-1} \PM{k}^{-4} D^{\mu} \PM{k} D_{\mu} D_{\nu} \PM{w} D^{\nu} \PM{w} \nonumber \\
& +\frac{1}{2} \Th^ {-1} \PM{k}^{-3} D_{\mu}\Delta \PM{w} D^{\mu} \PM{w} + 
\frac{1}{2} \Theta^{-1} \PM{k}^{-3} \cR_{\mu\nu} D^{\mu} \PM{w} D^{\nu} \PM{w} \nonumber \\
&  + \frac{1}{2} \Theta^{-1} \PM{k}^{-3}  D_{\mu}D_{\nu} \PM{w}D^{\mu} D^{\nu} \PM{w}  \nonumber \\
={}& \frac{1}{2} \PM{k}^{-7} \Theta^ {-1} (-4 \Theta^{-1} \PM{k}^7 \MP{w} D_{\mu} \PM{w} D^{\mu} \PM{k} - 
6 \PM{k}^6  D_{\mu} \PM{k} D^{\mu} \PM{k}  \nonumber  \\
&  - 2 \Theta^{-2} \MP{w}^2 \PM{k}^{12} + \PM{k}^4 D_{\mu} D_{\nu} \PM{w} D^{\mu} D^{\nu} \PM{w} ) 
 \nonumber \\
={}& \frac{1}{16} \Theta^3 \PM{k}^{-7} \PM{C}_{\mu\nu\si} \PM{C}^{\mu\nu\si},
\end{align}
where $ \Delta = D_{\mu} D^{\mu}$. The general case  $\al \in \mathbb{R}$ follows from \eqref{b1} by splitting the argument on the left hand side of \eqref{ddkw} as $\PM{k}^{\al + 1} / \PM{w}^{\al} =  \PM{k} (\PM{k} / \PM{w})^{\al}$ and applying the chain rule.
\end{proof} 

\begin{remark} \label{divrem}~
\begin{enumerate} 
\item \label{point:divrem:1}
The identity \eqref{ddkw} is equivalent to \eqref{eq:DivPsialpha} for $\alpha = 2\beta - 1$; hence remarks made in Section \ref{sec:divident} on the latter family for various values of $\beta$ carry over to the former.
\item
For $\al = 3$ the right hand side of \eqref{ddkw} simplifies to $\frac{1}{8} \Theta^{3}  \PM{B}_{\mu\nu} \PM{B}^{\mu\nu} / \PM{w}^{3}$.
\item 
When $\xi^{a}$ is hypersurface-orthogonal (which implies $\omega \equiv 0$, and $\P{w} = \M{w} = w$) the objects  $\PM{B}_{\mu\nu}$ coincide and we have
\begin{equation} 
\PM{B}_{\mu\nu} = B_{\mu\nu} = - 4 \Th^{-1} w^{-1} \left (1 \pm \Th^{1/2} \right) \PM{\cR}_{\mu\nu} 
\end{equation}
where $\PM{\cR}_{\mu\nu}$ are  the Ricci tensors with respect to the metrics $\PM{\gamma}_{\mu\nu}  =  \frac{1}{4} (\Th^{- 1/2} \pm 1)^2 \gamma_{\mu\nu}$.
\item
Still for $\omega \equiv 0$, each of $\PM{C}_{\mu\nu\si}$ reduces, via the field equations \eqref{lapw}, \eqref{Ric} to the Bach--Cotton tensor whose vanishing characterizes conformal flatness. The corresponding characterizations of the Lorentzian  Schwarzschild metric and the restriction of \eqref{ddkw} for certain values of $\al$ were employed in uniqueness proofs \cite{1967PhRv..164.1776I,MR398432}. For example, for $\al =0$ \eqref{ddkw} reduces  to a linear combination of \cite[Eq. (2.12)-(2.13)]{MR398432}
\item
In the general case $\PM{B}_{\mu\nu}$, $\PM{k}$ and $\PM{C}_{\mu\nu\si}$  have complex Lorentzian counterparts $B_{\mu\nu}$, $k$ and $C_{\mu\nu\si}$ which have analogous properties. They have been employed in local characterizations of the Kerr metric among the AF ones \cite{MR747036} and of larger classes of metrics if the asymptotic assumption is dropped, see  \cite{nozawa2021alternative} and references therein. 
\end{enumerate} 
\end{remark}

\subsection{The divergence identities on $\cM$}
We can now give an alternative proof of the divergence identitity \eqref{eq:DivPsialpha}.
\begin{proof} 
In suitable coordinates, both sides of \eqref{eq:DivPsialpha} are analytic on $\PM{U}_\eps$ and agree with the respective sides of \eqref{ddkw} on the  subset $\cM_0 \cap \PM{U}_\eps$ which is dense, open and connected by Theorem \ref{thm:quot}. Hence the validity of \eqref{eq:DivPsialpha} extends to $\PM{U}_\eps$ as required.
\end{proof}

\begin{remark} \label{quotint}
The family \eqref{ddkw} suggests an attempt for direct integration on $\cN_0$, without going the detour via $\cM$. In fact, this would directly yield relations \eqref{eq:magdim} which are just what is used in the uniqueness arguments. There are  problems with this strategy, however, which can be exemplified via the definition of the nut charge: Recall that on $\cM$, $N$ is defined as the flux of \eqref{eq:JT-def} through a hypersurface which may contain or intersect both maximal as well as exceptional orbits on $\cM$; accordingly $N$ it is not related in an obvious way to an analogous quantity defined from \eqref{divom} via a surface integral on $\cN_0$. This applies in particular to the key relation \eqref{nutgrav} for the surface gravities of the nut. To see this in concrete terms, relation \eqref{eq:magdim} would arise via integration of \eqref{divom} on $\cN$ with a ``quotient nut charge'' given by $N \kappa^1 /2\pi = 1/4 \kappa^2 $, where $|\kappa^1| \geq |\kappa^2|$, that is via division of $N$ by the length of the \emph{longer exceptional} orbit. In contrast, the ``quotient nut charge'' \cite[Eq. (5.9)]{gibbons:hawking:1979} is equal to  $N$ divided by the length of the \emph{maximal} orbit and thus reads $N \kappa^1 /2\pi w^1 = 1/(4 \kappa^2  w^1) =  1/(4 \kappa^1 w^2)$. In short, the integration procedure on $\cN$ is bound to fail without proper understanding of the integration process on the orbit space. We will not expose this in this work.
\end{remark}

\newcommand{\arxivref}[1]{\href{http://www.arxiv.org/abs/#1}{{arXiv:#1}}}
\newcommand{\prd}{Phys. Rev. D} 

\def\cprime{$'$}

\end{document}